\documentclass[12pt]{amsart}

\usepackage[T1]{fontenc}
\usepackage{amsmath}
\usepackage{amsthm}
\usepackage{amsfonts}
\usepackage{amssymb}
\usepackage{mathrsfs}
\usepackage{hyperref}
\usepackage{enumitem}
\usepackage{tikz}

 \usetikzlibrary{matrix,calc,arrows,knots,intersections,decorations.markings,cd}
\title{Dicritical divisors and hypercurvettes}
\author[E. Artal]{Enrique Artal Bartolo}
\address{Departamento de Matem\'aticas, IUMA\\
Universidad de Zaragoza\\
C.~Pedro Cerbuna 12\\
50009 Zaragoza, Spain}
\urladdr{\url{http://riemann.unizar.es/~artal}}
\email{\href{mailto:artal@unizar.es}{artal@unizar.es}}

\author[W. Veys]{Willem Veys}
\address{KU Leuven, Department of Mathematics\\
Celestijnenlaan 200B box 2400\\
BE-3001 Leuven, Belgium}
\urladdr{\url{https://perswww.kuleuven.be/~u0005725/}}
\email{\href{mailto:wim.veys@kuleuven.be}{wim.veys@kuleuven.be}}

 \subjclass[2010]{}


\thanks{The first named author is partially supported by MCIN/AEI/10.13039/501100011033 (grant code: PID2020-114750GB-C31)
and by Departamento de Ciencia, Universidad y Sociedad del Conocimiento del Gobierno de Arag{\'o}n
(grant code: E22\_20R: ``{\'A}lgebra y Geometr{\'i}a'').
The second named author is partially supported by KU Leuven grant GYN-E4282-C16/23/010.}

\newcommand\zz{\mathcal{Z}}
\newcommand{\ct}{\mathcal{A}}

\newcommand{\spl}{\mathcal{H}}

\newcommand{\grande}{\ell}
\newcommand{\vecl}{{\boldsymbol{\lambda}}}

\newcommand{\good}{\mathfrak{G}}
\newcommand{\poly}{W}

\newtheorem{thm}{Theorem}[section]
\newtheorem{thm0}{Theorem}
\newtheorem{lemma}[thm]{Lemma}
\newtheorem{prop}[thm]{Proposition}
\newtheorem{goal}[thm]{Goal}

\newtheorem{case}{Case}
\newtheorem{step}{Step}

\theoremstyle{definition}
\newtheorem{dfn}[thm]{Definition}
\newtheorem{example}[thm]{Example}

\newtheorem{notation}[thm]{Notation}

\theoremstyle{remark}
\newtheorem{remark}[thm]{Remark}

\newtheorem{remarks}[thm]{Remarks}

\DeclareMathOperator{\mat}{Mat}
\DeclareMathOperator{\codim}{codim}
\DeclareMathOperator{\divi}{div}

\begin{document}

\begin{abstract}
	Germs of rational functions~$h$ on points $p$ of  smooth varieties~$S$ define germs of rational maps
	to the projective line.
Assume that $p$ is in the indeterminacy locus of~$h$. 
	If $\pi:\hat{S}\to S$ is a birational map which is an isomorphism 	outside~$p$, then $h$ lifts to a germ of a rational
	map on $(\hat{S}, \pi^{-1}(p))$. The exceptional components $E_i$ of $\pi^{-1}(p)$ are classified according to the restriction of (the lift of) $h$ to $E_i$; the {\em 
	dicritical} components are those where this restriction induces a dominant map. In a series of papers, Abhyankar and the first named author studied this setting in dimension~$2$, where the main result is that, for any given $\pi$, there is a rational function $h$ with a prescribed  subset of exceptional components that are dicritical of some given degree.  

The concept of curvette of an exceptional component played a key role in the proof. 
The second named author extended previously the concept of curvette to the higher dimensional case. Here we use this concept to generalize the above result to arbitrary dimension. 
 \end{abstract}

\maketitle
\thispagestyle{empty}

\section*{Introduction}

The first named author studied with S.S.~Abhyankar~\cite{aba:13,aba:14} the relationship between curvettes
and rational functions with prescribed behavior in the two-dimensional case.
Let $p\in X_0$ be a smooth point of an algebraic 
complex surface. Let $f, g$ be two non-zero germs of polynomial functions at~$p$, both vanishing at~$p$. Then $h:=\frac{f}{g}$
defines a germ of a rational function which can be identified
with a germ of a such a map $h:(X_0,p)\dashrightarrow\mathbb{P}^1$ which is not defined at~$p$.
It is well-known that there is a sequence of point blow-ups such that for its composition
$\pi:(X_m, E)\to(X_0, p)$, the map~$h$ lifts to a
well-defined morphism $\tilde{h}:(X_m,E)\to\mathbb{P}^1$.

The exceptional components $E_1,\dots,E_m$ of $\pi$ have two different possible
behaviors with respect to $\tilde{h}$:
$\tilde{h}_{|E_i}$ is either constant or  surjective.
A component for which the latter holds is called \emph{dicritical}
and its degree is the one of
$\tilde{h}_{|E_i}:E_i\to\mathbb{P}^1$. In~\cite{aba:13,aba:14}
the following problem is solved. Starting from
a composition $\pi:(X_m, E)\to(X_0, p)$ of point blow-ups, and fixing
some of the exceptional components $E_{i_1},\dots,E_{i_r}$
together with some positive integers $d_{i_1},\dots,d_{i_r}$,
is there a rational function $h$ on $(X_0,p)$,
such that it can be lifted to a morphism $\tilde{h}:X_m\to\mathbb{P}^1$,
whose set of dicritical components is $\{E_{i_1},\dots,E_{i_r}\}$, and the degree
of the restriction of $\tilde{h}$ to $E_{i_j}$ is~$d_{i_j}$? 

The main tool to solve this problem is the use of~\emph{curvettes}.
A curvette $C_i$ of $E_i$ is a curve germ  at~$p$
whose strict transform in $X_m$ is a smooth curve, intersecting $E_i$ transversally at a point of $E_i\setminus\bigcup_{j\neq i} E_i$.
Considering \emph{generic}
families of $d_{i_j}$ curvettes for each $E_{i_j}$, the required
meromorphic (or rational) function $h=\frac{f}{g}$ can be constructed, taking $f$ and $g$ as suitable products of curvettes.

One of the motivations for that work, as well as for the present paper, comes from the study
of polynomial maps at infinity; the rational extension of these maps to the projective
space is not a morphism and indeterminacy points can be found at the hyperplane of
infinity. The local study of these rational functions at infinity involves birational modifications
and dicritical divisors are essential objects, see e.g.~\cite{arn, glm, lw}.

In this paper we focus on this problem in higher
dimension $n>2$.
We consider again a composition of admissible blow-ups $\pi:(X_m, E)\to(X_0, p)$ that is an isomorphism outside~$p$, hence starting with blowing up~$p$.
Now it is not possible
to lift a rational function~$h$ through
 $\pi$ to a {\em morphism} $X_m\to\mathbb{P}^1$, since the subvariety of indeterminacy
points of~$h$ is of positive dimension.
Nevertheless, the notion of dicritical component
can be define accordingly and a similar problem
can be stated.

\smallskip
The notions of curvette and dicritical divisors have been used 
in several works, e.g. \cite{cae:10, oku, AL:11, AH:12, clmn}, with analytic or algebraic flavor, in smooth or singular ambient spaces but always in the realm of  dimension~$2$.

There are several possible generalizations
of the notion of curvette when $n>2$. The simplest one
is to consider one-dimensional curvettes as in the surface
case with a similar definition. The problem appears
if one wants to define \emph{hypercurvettes};
there is not a canonical notion of higher dimensional
curvette of some~$E_i$, as a \emph{ hypersurface} with \emph{nice} intersections with $E_i$.
In~\cite{vy:07}, the second
named author proposes such a hypersurface generalization.

 There are several intrinsic differences with the
$2$-dimensional case. First, their \emph{equisingularity} class is not determined by~$E_i$ (as in the surface case); there are several admissible choices.  
Next, in the surface case, there is a notion of intersection
matrix which is an important tool, namely the $(m\times m)$ matrix
of intersection numbers of  curvettes $C_1,\dots, C_m$.
This intersection matrix has only positive entries, it is
symmetric and positive definite, and all its principal submatrices are
unimodular. For $n>2$, a similar intersection matrix
can be defined, considering the intersection numbers
of a one-dimensional general curvette for $E_i$
and a (codimension one)
general hypercurvette for $E_j$. This new intersection matrix
is neither unique nor symmetric, but it has only
positive entries and the above property on its principal submatrices
is preserved.

Another crucial difference is the intersection of a hypercurvette with the exceptional locus.
In the two-dimensional case, a curvette only intersects its associated
exceptional curve (in one point) and this fact simplifies many arguments.  But in higher dimensions a hypercurvette of $E_i$
intersects in general {\em many} other exceptional components $E_\ell$. This creates various difficulties, similar to the ones encountered in the study of filtrations associated to singularities~\cite{vn:24}.
This notion of higher dimensional
curvette is a main ingredient in this work.

\smallskip
Another difference with the two-dimensional
case is that the notion of degree of a dicritical
component in a composition of blow-ups
depends on the order in which these
blow-ups are constructed; in Example~\ref{ex1}
we present a composition of three blow-ups
over $0\in\mathbb{C}^3$ with two essentially
different orderings, giving rise to such a difference in degree.

The main result of the paper is the following one.
Given a smooth point $p$ in a complex algebraic variety $X_0$,
we fix a modification $\pi$ of $X_0$ which is an isomorphism
outside~$p$, obtained as a composition of blow-ups with admissible center, where the first blow-up has center $p$.
Let $E_1,\dots,E_m$ denote the exceptional components of $\pi$.

\begin{thm0}\label{thm:main}
Let $\emptyset \neq J\subset\{1,\dots,m\}$ and let $d_j\in\mathbb{Z}_{\geq 1}$, $j\in J$.
Then there exists a rational function $h$ for which $E_j$ is  dicritical of degree~$d_j$ if $j\in J$,
and $E_j$ is
non-dicritical  if $j\notin J$.
\end{thm0}

In the two-dimensional case, the starting point
is to prove this result for one dicritical and $d_1=1$; a product
of such functions (with positive and negative exponents) then gives the desired result.
Using products of equations of hypercurvettes
we obtain
a similar result in Proposition~\ref{prop:one-dic},
but the statement is weaker, since we do not control the degree
of the dicritical.

In Proposition~\ref{prop3} we already make substantial progress, since there we can prescribe the degree, but only for
the \emph{last} dicritical. More precisely, we can ensure
the existence of a rational function for which a fixed $E_s$, $s\leq m$,
is a dicritical component with prescribed degree and
$E_j$, $j<s$, are not dicritical components, but we loose
control on $E_j$, $j>s$. This issue is solved in
Theorem~\ref{thm4}. The way to the proof of Theorem \ref{thm:main} is then quite
straightforward.

\smallskip
In \S\ref{sec:settings} we introduce the main objects, as
the sequence of blow-ups of \eqref{sequence}
and the higher dimensional curvettes
of Proposition~\ref{hypercurvettes}, following~\cite{vy:07}.
Some more tools and mainly illustrating examples are
exhibited. In \S\ref{sec:rational} the problem is stated
together with the first result,
Proposition~\ref{prop:one-dic}. In \S\ref{sec:special}
we introduce the notion of \emph{special hypersurface}
associated to an exceptional component. These hypersurfaces
are used to \lq compensate\rq\ the effect of the hypercurvettes used
in the proof of Proposition~\ref{prop:one-dic}, more precisely in \S\ref{sec:one_dic} to prove
Proposition~\ref{prop3}. Example~\ref{example_special}
shows however that we may have problems with the last exceptional
components.  Theorem~\ref{thm4} solves these issues
and we show with some examples the idea behind the proof.
We end this section with the proof of the main theorem.
The long proof of Theorem~\ref{thm4} is carried out in the last
section \S\ref{sec:proof_thm}.


\section{Settings}\label{sec:settings}
\numberwithin{equation}{section}

In this paper we deal with complex algebraic varieties, where
the concept \emph{variety} means irreducible algebraic set.
Given a smooth point $p$ in a complex algebraic variety $X_0$ of dimension $n$,
we fix a modification $\pi$ of $X_0$ which is an isomorphism
outside~$p$. More precisely, we consider a chain
of blow-ups with admissible centers
\begin{equation}\label{sequence}
\begin{tikzcd}
X_0\arrow[<-,r,"\pi_1"]&X_1\ar[<-,r,"\pi_2"]&\dots
\ar[<-,r,"\pi_{m-1}"]&X_{m-1}\ar[<-,r,"\pi_{m}"]&X_m,
\arrow[llll,"\pi" above,bend left=30]
\end{tikzcd}
\end{equation}
where $\pi_1$ is the blow-up at $p$.
Let us denote by $Z_i\subset X_{i-1}$ the center of $\pi_i$ and by $E_i\subset X_i$ the exceptional divisor of~$\pi_i$.
The consecutive strict transforms of $E_i$ by $\pi_j, j>i$, are still denoted by $E_i$, in order not to overload the notation.
We thus have $Z_1=p$, and for $j>1$ we assume that the center~$Z_j$ of $\pi_j$ is contained
in $\bigcup_{i<j} E_i$ (the exceptional locus of $\pi_{j-1}$), and has normal crossings with it.
We put also
\[
\pi_{j,i}=\pi_{i+1}\circ\cdots\circ\pi_j:X_j\to X_i,\quad
j>i.
\]
The divisors $E_1,\dots,E_m$ in $X_m$ define divisorial
valuations $\nu_1,\dots,\nu_m$, respectively. We will use these valuations for rational functions
and divisors
on $X_m$.


While we are mainly interested in the composition~$\pi$, most of the constructions depend on the sequence of blow-ups.
When $E_j$ is created, it can be expressed as a bundle
\[
\mathbb{P}^{n_j}\hookrightarrow
E_j\xrightarrow{\quad\sigma_j\quad} Z_j,\quad n_j:=n-\dim Z_j-1.
\]
We define a special class of curves in~$E_j$,
namely the \emph{general} lines~$\ell_j$ in \emph{general} fibers of the above projective space bundle.
Note that when $Z_j$ is a point, $\ell_j$ is just a general line in a projective space $\mathbb{P}^{n-1}$.
When $n=3$ and $Z_j\cong\mathbb{P}^1$, $E_j$ is a ruled surface
having a section with self-intersection~$-a$,
i.e. isomorphic to a Hirzebruch surface $\Sigma_a$
for some $a\geq 0$.

\begin{remark}
Note that in $X_m$ the divisor $E_j$ is in general a blown-up of the original~$E_j$. If we choose another sequence of blow-ups for which
the final result coincides, the class of the curves~$\ell_j$ may be different.
\end{remark}

\begin{example}\label{ex1}
Let $X_1\xrightarrow{\pi_1} X_0=\mathbb{C}^3$ be the blow-up of the origin, with exceptional component $E_1\cong\mathbb{P}^2$. The class $\ell_1$ is given by a general line in $E_1$. We pick  a line $L\subset E_1$ and a point $P\in L$
(as $\ell_1$ is general $P\notin\ell_1$).

Let
$X_2\xrightarrow{\pi_2} X_1$ be the blow-up of~$P$.
Then $E_2\subset X_2$ is isomorphic to $\mathbb{P}^2$ and
$E_1\subset X_2$ is the blow-up of a point in $\mathbb{P}^2$,
i.e., isomorphic to $\Sigma_1$. Let $E_{12}:=E_1\cap E_2$; it is a line in $E_2$ and the negative section in~$E_1$, see Figure~\ref{fig:2blowups}.
The class~$\ell_1$ corresponds to the $(+1)$-sections of~$\Sigma_1$,
while $\ell_2$ is a general line in~$E_2$. The line~$L$ becomes a fiber of~$E_1$ and $L\cap E_{12}=L\cap E_2$ is a point.

\begin{figure}[ht]
\begin{center}
\begin{tikzpicture}[yscale=.5]
\fill (-2,0) ellipse [x radius=.075cm, y radius=0.15cm] node[left] {$\mathbf{0}\in\mathbb{C}^3$};
\draw[<-] (-1.8,0)--(0,0) node[above,pos=.5] {$\pi_1$};
\draw (0,-1)--(3,-1)--(4,1)--(1,1)--cycle node[below=5pt] {$E_1\cong\mathbb{P}^2$};
\draw (.75,0)-- node[above] {$P$} (3.25,0) node[above] {$L$};
\fill (2,0) ellipse [x radius=.075cm, y radius=0.15cm];
\draw[dashed]  (1,-.75)--  node[above right=-3pt] {$\ell_1$} (1,.75);
\draw[<-] (4,0)--(6,0) node[above,pos=.5] {$\pi_2$};

\begin{scope}[xshift=6cm,yshift=1.5cm]
\draw (0,-1)--(3,-1)--(4,1)--(1,1)--cycle;
\node[right] at (3,-1) {$E_1\cong\Sigma_1$};
\draw (.75,0)-- (3.25,0) node[above] {$L$};
\draw (1.5,-1)-- (2.5,1) node[above] {$E_{12}$};
\fill (2,0) ellipse [x radius=.075cm, y radius=0.15cm];
\draw[dashed]  (1,-.75)--  node[above right=-3pt] {$\ell_1$} (1,.75);
\end{scope}

\begin{scope}[xshift=6cm,yshift=-1.5cm]
\draw (0,-1)--(3,-1)--(4,1)--(1,1)--cycle;
\node[right] at (3,-1) {$E_2\cong\mathbb{P}^2$};
\draw (1.5,-1) --  (2.5,1) node[below right=-3pt] {$E_{12}$} ;
\fill (2,0) ellipse [x radius=.075cm, y radius=0.15cm];
\draw[dashed]   (1,-.5)  node[left] {$\ell_2$} --(3,-.5);
\end{scope}
\end{tikzpicture}

\caption{$\pi_2\circ\pi_1$}
\label{fig:2blowups}
\end{center}
\end{figure}
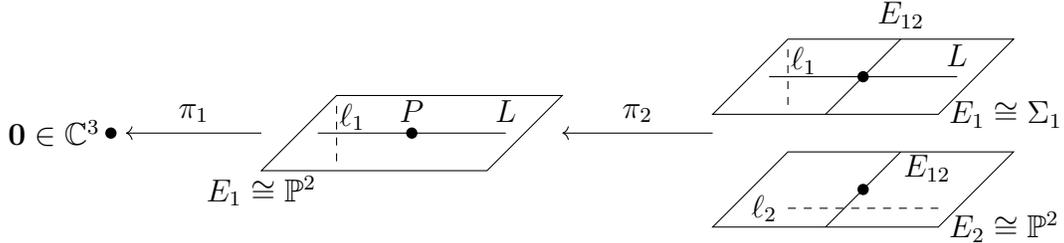

We take as $X_3\xrightarrow{\pi_3} X_2$ the blow-up along $L$.
We summarize the result:
\begin{itemize}
\item $E_3\subset X_3$ is isomorphic to $\Sigma_2$, and $\ell_3$ is a general fiber;
\item $E_2\subset X_3$ is isomorphic to the blow-up of a point in $E_2\subset X_2$, i.e., it is isomorphic to $\Sigma_1$.
The class $\ell_2$ is a general $(+1)$-section in $\Sigma_1$. Note that $E_{23}:=E_2\cap E_3$ is the
$(-1)$-section of $E_2$ and a fiber in~$E_3$;

\item $E_1\subset X_3$ is isomorphic to $\Sigma_1$ and $\ell_1$ is a general $(+1)$-section. Note that $E_{13}:=E_1\cap E_3$ is a fiber in $E_1$ and the $(-2)$-section in $E_3$.
\end{itemize}

It is possible to express $\pi:X_3\to X_0$ as another sequence
of blow-ups $\bar{\pi}_1\circ\bar{\pi}_2\circ\bar{\pi}_3$, where
$\pi_1=\bar{\pi}_1$. The map
$\bar{\pi}_2:\bar{X}_2\to\bar{X}_1=X_1$ is the blow-up
along~$L$. Hence $\bar{E}_2$ is isomorphic to~$\Sigma_2$
and $\bar{\ell}_2$ is a general fiber. The curve $\bar{E}_{12}:=E_1\cap E_2$ is the $(-2)$-section of $\bar{E}_2$
and a line in $E_1$ (which is still isomorphic to $\mathbb{P}^2$).
The preimage of the point~$P$ by $\bar{\pi}_2$ is a fiber $F$ of $\bar{E}_2$, see Figure~\ref{fig:2blowups2}.

\begin{figure}[ht]
\begin{center}
\begin{tikzpicture}[yscale=.5]
\fill (-2,0) ellipse [x radius=.075cm, y radius=0.15cm] node[left] {$\mathbf{0}\in\mathbb{C}^3$};
\draw[<-] (-1.8,0)--(0,0) node[above,pos=.5] {$\bar\pi_1$};
\draw (0,-1)--(3,-1)--(4,1)--(1,1)--cycle node[below=5pt] {$\bar E_1\cong\mathbb{P}^2$};
\draw (.75,0)-- node[above] {$P$} (3.25,0) node[above] {$L$};
\fill (2,0) ellipse [x radius=.075cm, y radius=0.15cm];
\draw[dashed]  (1,-.75) --  node[below right=-4pt] {$\bar\ell_1$} (1,.75);
\draw[<-] (4,0)--(6,0) node[above,pos=.5] {$\bar\pi_2$};

\begin{scope}[xshift=6cm,yshift=1.5cm]
\draw (0,-1)--(3,-1)--(4,1)--(1,1)--cycle;
\node[right] at (3,-1) {$\tilde E_1\cong\mathbb{P}^2$};
\draw (.75,0)-- (3.25,0) node[above] {$\bar E_{12}$};
\fill (2,0) ellipse [x radius=.075cm, y radius=0.15cm];
\draw[dashed] (1,-.75) -- node[below right=-4pt] {$\bar\ell_1$} (1,.75);
\end{scope}

\begin{scope}[xshift=6cm,yshift=-1.5cm]
\draw (0,-1)--(3,-1)--(4,1)--(1,1)--cycle;
\node[right] at (3,-1) {$\bar E_2\cong\Sigma_2$};
\draw (.75,0)  node[above right] {$F$} --  (3.25,0);
\draw (1.5,-1) --  (2.5,1)  node[below right=-3pt] {$\bar E_{12}$} ;

\fill (2,0) ellipse [x radius=.075cm, y radius=0.15cm];
\draw[dashed]   (1,-.5)  node[left] {$\bar \ell_2$} --(3,-.5);
\end{scope}
\end{tikzpicture}
\caption{$\bar\pi_2\circ\bar\pi_1$}
\label{fig:2blowups2}
\end{center}
\end{figure}
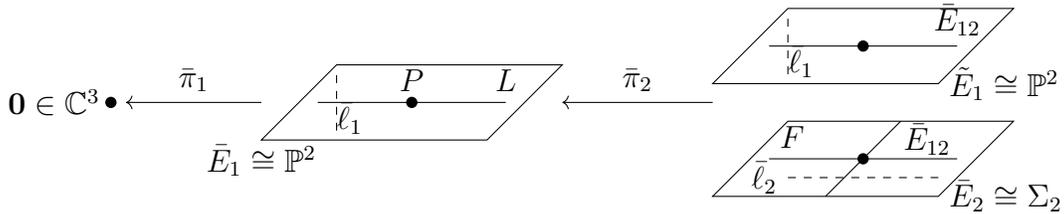

Let $\bar{\pi}_3:\bar{X}_3\to\bar{X}_2$ be the blow-up
along~$F$. It is not hard to check that $X_3=\bar{X}_3$, and that in $X_3=\bar{X}_3$ we have
$E_1=\bar{E}_1$ and $E_j=\bar{E}_k$, $\{j,k\}=\{2,3\}$.
Moreover, we have $\bar{\ell}_1=\ell_1$ and $\bar{\ell}_2=\ell_3$ (as classes) but $\ell_2$ is a general $(+1)$-section while $\bar{\ell}_3$
is a fiber of $\bar E_3=E_2\cong\Sigma_1$, see
Figure~\ref{fig:3blowup}.

\begin{figure}[ht]
\begin{center}
\begin{tikzpicture}[yscale=.5]
\draw (0,-1)--(3,-1)--(4,1)--(1,1)--cycle;
\node[below right] at (0,-1) {$\bar E_1=E_1\cong\Sigma_1$};
\draw (.75,0)  -- (3.25,0) node[below left=-3pt] {$E_{13}$};
\draw (1.5,-1)-- (2.5,1) node[above] {$E_{12}$};
\fill (2,0) ellipse [x radius=.075cm, y radius=0.15cm];
\draw[dashed]  (1,-.75)--(1,.75) node[left=4pt] {$\ell_1=\bar\ell_1$};

\begin{scope}[xshift=4cm]
\draw (0,-1)--(3,-1)--(4,1)--(1,1)--cycle;
\node[below right] at (0,-1) {$\bar E_3=E_2\cong\Sigma_1$};
\draw (.75,0)-- (3.25,0) node[above=-3pt] {$E_{12}$};
\fill (2,0) ellipse [x radius=.075cm, y radius=0.15cm];
\draw (1.5,-1)-- (2.5,1) node[above] {$E_{23}
$};
\draw[dashed]  (1.25,-.75)--  node[above right=-3pt] {$\ell_2$} (1.25,.75);
\draw[dashed]   (1,-.5)  node[left] {$\bar\ell_3$} --(3,-.5);
\end{scope}

\begin{scope}[xshift=8cm]
\draw (0,-1)--(3,-1)--(4,1)--(1,1)--cycle;
\node[below right] at (0,-1) {$\bar E_2=E_3\cong\Sigma_2$};
\draw (.75,0)--  (3.25,0) node[above=-3pt] {$E_{23}$};
\draw (1.5,-1)-- 
(2.5,1) node[above] {$E_{13}$};
\fill (2,0) ellipse [x radius=.075cm, y radius=0.15cm];
\draw[dashed]   (1,-.5) --(3,-.5)  node[right=5pt] {$\ell_3=\bar\ell_2$} ;
\end{scope}
\end{tikzpicture}
\caption{$X_3=\bar{X}_3$}
\label{fig:3blowup}
\end{center}
\end{figure}

\end{example}

Let us recall the notion of \emph{higher dimensional curvettes}
introduced by the second named author.

\begin{prop}[{\cite[Proposition 3.2]{vy:07}}]\label{hypercurvettes}
Consider the \emph{ordered} modification $\pi$ as in \eqref{sequence}.
One can construct consecutively
for $j = 1,\dots,m$ a hypersurface $\mathcal{C}_j$ on $X_0$ with the following properties.

\begin{enumerate}[label=\rm(\arabic{enumi})]
\item The strict transform of $\mathcal{C}_j$  in $X_{j-1}$ contains $Z_j$ and  is smooth along $Z_j$.

\item Denoting by $\tilde{\mathcal{C}_i}$, $i\leq j$,  the strict transform  in $X_j$ of
$\mathcal{C}_i$, we have that  $E_1 \cup E_2 \cup \cdots \cup E_j \cup\tilde{\mathcal{C}_1} \cup \tilde{\mathcal{C}_2} \cup \cdots \cup \tilde{\mathcal{C}_j}$ is a normal crossing divisor on $X_j$. Also, the next center of blow-up $Z_{j+1}$ has normal crossings with $\tilde{\mathcal{C}_1} \cup \tilde{\mathcal{C}_2} \cup \cdots \cup
\tilde{\mathcal{C}_j}$, and moreover $Z_{j+1}$ is not contained in this union.
(In particular, if $Z_{j+1}$ is a point, it does not belong to any $\tilde{\mathcal{C}_i}, i \leq j$.)

\item\label{hypercurvettes2} For the $\mathbb{P}^{n_j}$-fiber bundle  $\sigma_j$
 with total space
 $E_j\subset X_j$,
we have that the intersection multiplicity of $\tilde{\mathcal{C}_j}\cap E_j$  with the general line $\ell_j\subset\mathbb{P}^{n_j}$ is $1$.
(In particular, when $E_j$ is a projective space, then $\tilde{\mathcal{C}_j}\cap E_j$ is a general hyperplane in $E_j$.)

\end{enumerate}

In addition, the construction allows to associate to each $E_j$ any finite number of such $\mathcal{C}_j$, such that in
\ref{hypercurvettes2} 
all exceptional components and all strict transforms together form a normal crossing divisor.


\end{prop}


The total transforms in $X_m$ of each
 $\mathcal{C}_j$ can be expressed as a divisor as follows:
\begin{equation}\label{eq:matrix}
\pi^*\mathcal{C}_j=\tilde{\mathcal{C}_j}+\sum_{i=1}^m a_{ji} E_i,\qquad a_{ji}=\nu_i(\mathcal{C}_j).
\end{equation}

\begin{dfn}
The \emph{valuation matrix} of this system of hypercurvettes is the matrix
$A\in\mat(m;\mathbb{Z})$ of the $a_{ji}$.
\end{dfn}

This matrix is in fact precisely the new intersection matrix for $n>2$ that we mentioned in the introduction.  Let $c_i$ be a one-dimensional curvette of $E_i$, that is, a curve germ through $p$ whose strict transform $\tilde{c}_i$ in $X_m$ intersects $E_i$ transversally in a generic point.
Let $I\in\mat(m;\mathbb{Z})$ denote the  matrix consisting of the $\mathcal{C}_j \cdot c_i$.

\begin{lemma}
With the notations above, we have that $A=I$.
\end{lemma}

\begin{proof} Note that, by definition of $c_i$, we have that  $E_\ell \cdot \tilde{c}_i = \delta_{\ell i}$. Moreover, by our generality assumptions, $\tilde{c}_i$ does not intersect any  $\tilde{\mathcal{C}_j}$. Then the projection formula yields
\[
\mathcal{C}_j \cdot c_i = \mathcal{C}_j \cdot \pi_* \tilde{c}_i = (\pi^*\mathcal{C}_j )\cdot \tilde{c}_i =( \tilde{\mathcal{C}_j}+\sum_{\ell=1}^m a_{j\ell} E_\ell)\cdot \tilde{c}_i = a_{ji}.
\qedhere
\]
\end{proof}

\begin{remark}
When $n=2$, the classical curvettes $\mathcal{C}_j$ satisfy the properties of Proposition \ref{hypercurvettes}, and for those the matrix $A$ is \lq canonical\rq; it is the inverse of minus the intersection matrix of the curves $E_1, \dots, E_m$ in $X_m$.  In particular, this matrix is symmetric.

When $n \geq 3$, the $\mathcal{C}_j$ above depend on some choices, and in particular on the sequence of blow-ups $\pi_j$ that constitute $\pi$.
There does not seem to be any canonical choice for such $\mathcal{C}_j$.
Most importantly, whatever choice is made,  in general in $X_m$ the hypersurface $\tilde{\mathcal{C}_j}$ intersects not only $E_j$, but also  components $E_i, i\neq j$, and this does not happen for $n=2$.

 In particular, also the matrix $A$ depends on choices, and it is in general not symmetric.
However, an important property of the two-dimensional setting does generalize.
\end{remark}

\begin{prop}[{\cite[Proposition 3.3]{vy:07}}]\label{matrixA}
For each $t=1,\dots,m$, let $A_t$ be the principal submatrix of $A$, formed by the first $t$ rows and columns. We have that $\det A_t=1$ for all $t$.
\end{prop}

\begin{example}[Continuation of  Example \ref{ex1}]
The surface $E_2$ is created as $\mathbb{P}^2$, hence for any choice of \emph{higher dimensional curvette} associated to $E_2$, its strict transform in $X_2$ intersects $E_2$ in a generic line. And then in $X_3$ it intersects $E_2\cong \Sigma_1$ in a generic $(+1)$-section.
On the other hand, the corresponding $\bar{E}_3$ in $\bar{X}_3=X_3$ is created as $\Sigma_1$. And then the intersection of $E_2=\bar{E}_3\cong\Sigma_1$ and a \emph{higher dimensional curvette} associated to $\bar{E}_3$  is some section of $\Sigma_1$ (non necessarily a $(+1)$-section).

Let us illustrate this with equations.
We use coordinates $x,y,z$ in $\mathbb{C}^3=X_0$, as well as in the various charts of $X_i, i>0$, in the usual way.
Here we fix them such that, in the chart of $X_1$ where $E_1$ is given by $z=0$, the point $P$ is the origin and the line $L$ is given by $y=z=0$. For the next $X_i$ we choose the charts
where non-trivial intersections arise.
On can verify that the following are possible choices for \emph{higher dimensional curvettes}\ associated to the modification $\pi$:
\[
\mathcal{C}_1 : a_1x+b_1y+c_1z=0,\quad  \mathcal{C}_2: a_2x+b_2y=0,\quad  \mathcal{C}_3: b_3y +c_3z^2 =0,
\]
where all coefficients $a_i,b_i,c_i$ are generic.
The associated matrix $A$ is
\[
A=
\begin{pmatrix}
1 & 1 & 1 \\
1 & 2 & 1 \\
1 & 2 & 2 \\
\end{pmatrix}.
\]
Then, for instance in the chart of $X_3$ where $E_1$, $E_2$ and $E_3$ are given by $z=0$, $x=0$ and $y=0$, respectively, the strict transforms of the $\mathcal{C}_j$ are given as
\[
\tilde{\mathcal{C}}_1 : a_1x+b_1xy+c_1=0,\quad \tilde{\mathcal{C}}_2: a_2+b_2y=0,\quad  \tilde{\mathcal{C}}_3: b_3 +c_3z =0.
\]
In particular, $\tilde{\mathcal{C}}_2 \cap E_2$ is disjoint from $E_3 \cap E_2$ (which is also true in the other charts).  This is mandatory since these curves are a $(+1)$-section and the $(-1)$-section, respectively, on $E_2\cong \Sigma_1$.

For the modification $\bar{\pi}$, let us take the same $\bar{\mathcal{C}}_1 =\mathcal{C}_1$ associated to $\bar{E}_1=E_1$ and $\bar{\mathcal{C}}_2 = \mathcal{C}_3$ associated to $\bar{E}_2=E_3$.
For $\bar{E}_3=E_2$, we now choose
\[
\bar{\mathcal{C}}_3: dy^2 +exz^2 + fxy =0,
\]
with generic coefficients $d,e,f$.
The associated matrix $\bar A$ is
\[
\bar A=
\begin{pmatrix}
1 & 1 & 1 \\
1 & 2 & 2 \\
2 & 3 & 4 \\
\end{pmatrix}.
\]
Then in the similar chart of $\bar{X}_3=X_3$,  we have that  $\bar{E}_1$, $\bar{E}_2$ and $\bar{E}_3$ are given by $z=0$, $y=0$ and $x=0$, respectively, and the strict transform of $\bar{\mathcal{C}}_3$ by $dy+ez+f =0$.
In particular, its intersection with $\bar{E}_3$ intersects $\bar{E}_2 \cap \bar{E}_3$ (transversely), and hence is not a $(+1)$-section of $\Sigma_1$.
\end{example}

From now  on, we fix a choice of (classes of) $\mathcal{C}_j$ as in Proposition \ref{hypercurvettes}. For the sake of simplicity such \emph{higher dimensional curvettes}\
will be called \emph{hypercurvettes}.


\begin{notation}
\mbox{}
\begin{enumerate}
\item Let~$X$ be an 
algebraic variety.
For an open and dense affine subspace $U$ of $X$, we denote by $R[U]$ the ring of regular functions of~$U$ and by $R(U)=R(X)$ its field of rational functions. If $Z\subset X$ is an irreducible Zariski closed subset, $R_Z$ denotes the local ring of $X$ at (the generic point of) $Z$. If $Z\cap U\neq\emptyset$, then $R_Z$ can be described as the localization of $R[U]$ with respect to the prime ideal of the functions vanishing at $Z\cap U$.

\item In order not to overload notation, we will in general identify a hypersurface
and its defining regular function, in particular for hypercurvettes.
\end{enumerate}

\end{notation}

\section{Rational functions}\label{sec:rational}
\numberwithin{equation}{section}



%

Let $h:X_0\dashrightarrow\mathbb{C}$ be a non-constant rational
function whose indeterminacy locus contains $p$, i.e., the quotient of two non-proportional regular functions both vanishing
at~$p$.

%
%

\begin{dfn}
Let $h:(X_0,p)\dashrightarrow\mathbb{C}$ be a non-constant rational function
and let $\pi:X_m\to X$ be a sequence of blow-ups
as in \eqref{sequence}. Consider the exceptional divisor $E_i \ (\subset X_i)$, seen as a projective bundle over $Z_i$.
\begin{enumerate}[label=(\arabic*)]
 \item The divisor $E_i$ is \emph{constant} or \emph{non-dicritical}, if the pull-back of $h$, restricted to $E_i$,
 is constant (this constant
 can be a value in $\mathbb{C}\cup\{\infty\}$).

 \item The divisor $E_i$ is \emph{dicritical} if the pullback of $h$, restricted to  $E_i$,
 is a dominant rational map. The \emph{degree}
 of $E_i$ is the intersection number of a general fiber of $(\pi^*h)|_{E_i}$ with $\ell_i$. 
 Note that a dicritical can be of degree~$0$.
\end{enumerate}
\end{dfn}


Let $h:X_0\dasharrow\mathbb{C}$ be a non-constant rational function. For any $i=1,\dots,m$ we can consider $\pi_i^*h$ and its divisor
\[
\divi\pi_i^*h=(\text{strict transform of} \divi h) + \sum_{j=1}^i N_j E_j
\]
in $X_i$. We put only one subindex on $N_j$, since
\[
\divi\pi_{i+1}^*h=(\text{strict transform of} \divi h) + \sum_{j=1}^{i+1} N_j E_j
\]
in $X_{i+1}$. Actually, $N_i=\nu_i(h)$. The following well-known formula for $N_{i+1}$ allows to compute all these multiplicities inductively.

\begin{lemma}\label{lem:inductive}
Let $S_i$ be the strict transform divisor part of $\divi\pi_i^*h$ in $X_i$. Let $t_i$ be the
multiplicity of $S_i$ at $Z_{i+1}$. Then,
\[
N_{i+1} = t_i + \sum_{j\in D_{i+1}} N_j,\quad\text{ with }
D_{i+1}:=\{j\in\{1,\dots,i\}\mid Z_{i+1}\subset E_j\}.
\]
\end{lemma}

\begin{example}\label{ejm:matrix}
We can deduce the upper triangular part of the matrix $A$ inductively, see~\cite[Proposition 3.3]{vy:07}, starting from $a_{11}=1$.  Let $k\in\{2,\dots,m\}$; then
\begin{equation}\label{eq:rels_A0}
a_{ik}=\sum_{j\in D_{k}} a_{ij},\quad 1\leq i<k,\qquad\text{ and }
a_{kk}=\sum_{j\in D_{k}} a_{kj} +1.
\end{equation}
In fact,
\[
a_{ik}=\nu_k(\mathcal{C}_i)=t_i + \sum_{j\in D_{k}} \nu_j(\mathcal{C}_i)= t_i+\sum_{j\in D_{k}} a_{ij},
\]
where $t_i$ is the multiplicity of the strict transform of $\mathcal{C}_i$ at $Z_k$:
if $i<k$, then $t_i=0$, and if $i=k$, then $t_i=1$.
\end{example}

\begin{example}\label{ejm:dicritical}
This example reflects some features which do not appear
in the previous ones, in particular the existence
of dicritical components of degree~$0$.
We consider a sequence of blow-ups $\pi:X_2\to X_0=\mathbb{C}^3$ defined as follows. As it is compulsory, $\pi_1:X_1\to X_0$ is the blow-up of the origin and hence $E_1\cong\mathbb{P}^2$. Let $C_d\subset E_1$ be a smooth curve of degree~$d$ and let $\pi_2:X_2\to X_1$ be the blow-up along~$C_d$.

In $X_2$, we have that $E_1$ is still~$\mathbb{P}^2$, while $E_2$ is a ruled surface
with base the smooth curve $C_d$ of genus~$\frac{(d-1)(d-2)}{2}$. Let
$E_{12}:=E_1\cap E_2$; as a curve in $E_1$ it is a copy of $C_d$,
and as a curve in $E_2$ it is a section with self-intersection $-d(d+1)$.
Recall that $\ell_1$ is a general line in $E_1$ (intersecting $C_d$ at $d$ points) and $\ell_2$ is a fiber of the ruled surface $E_2$.

\smallskip
Let now $h$ be a rational function on $X_0=\mathbb{C}^3$
and let us study what happens in a neighbourhood
of $p=0$. We can write
$h=\frac{f_0}{f_\infty}$, where $f_0$ and $f_\infty$ are polynomials. Assume that
\[
f_0=P_k+\sum_{i>k} P_i,\quad f_\infty=Q_k+\sum_{i>k} Q_i \quad
\text{(decomposition in homogeneous parts),}
\]
for some  $P_k,Q_k$ of degree $k\geq 1$. Denote $G:=\gcd(P_k,Q_k)$.

\begin{enumerate}[label=\rm(\arabic{enumi})]
\item\label{case251}  If $C_d \nmid G$, then $E_1$ is a dicritical of degree~$k$ for~$h$.
If furthermore $C_d$ is a component of a curve
in the pencil generated by $P_k,Q_k$, 
then $E_2$ is constant for~$h$. If this is not the case, then $E_2$ is a dicritical of degree~$0$ for~$h$.

\item  If $C_d \mid G$, more precisely $C_d^\ell\mid G$ but $C_d^{\ell +1} \nmid G$, then $E_1$ is a dicritical of degree~$k-\ell d$ for~$h$.
Concerning $E_2$, there is a similar distinction as in case \ref{case251}, with respect to the pencil generated by $\frac{P_k}{C_d^\ell},\frac{Q_k}{C_d^\ell}$.
\end{enumerate}

\end{example}



\smallskip
We recall our main goal, and derive some easy partial result.  In the next sections, we develop techniques to achieve the goal in complete generality.

\begin{goal}\label{goal}
Given a sequence of blow-ups $\pi:X_m\to X$ as in \eqref{sequence}, a tuple $(E_{i_1},\dots,E_{i_r})$ of exceptional
components and \emph{positive} degrees~$(d_{i_1},\dots,d_{i_r})$, construct a rational function $h:(X_0,p)\dasharrow\mathbb{C}$ such that the  dicritical components of $h$ are precisely $(E_{i_1},\dots,E_{i_r})$
with degrees~$(d_{i_1},\dots,d_{i_r})$.
\end{goal}


\begin{notation}\label{notation:ri}
The notation $\mathcal{C}_i^{(r_i)}$, $r_i\in\mathbb{Z}$, will stand 
 for a product of factors $\mathcal{C}_{i,\ell}^{\varepsilon_\ell}$, $\varepsilon_\ell\in \mathbb{Z}$, with $r_i=\sum_\ell\varepsilon_\ell$, for different
and generic hypercurvettes of the same exceptional component $E_i$.
\end{notation}

Recall that we assume normal crossings for the union of all hypercurvettes and all exceptional components.
The verification of the following properties is straightforward.

\begin{lemma}\label{lem:numer_dic}
Let $h:=\prod_{j=1}^m \mathcal{C}_j^{(r_j)}$. Let $N_i=\sum_{j=1}^m r_j a_{ji}\in \mathbb{Z}$ denote the multiplicity of $\pi^*(h)$ along~$E_i$.
\begin{enumerate}[label=\rm(\arabic*)]
\item If $N_i\neq 0$, then $E_i$ is not a dicritical for $h$.
\item If $N_i=0$ and $r_i\neq 0$, then $E_i$ is a dicritical for $h$ of degree~$\geq 1$.
\item\label{lem:numer_dic3} If $N_i=0$, $r_i= 0$, and $\mathcal{C}_i^{(r_i)}$ is a non-constant function, then $E_i$ is a dicritical for $h$ of degree~$\geq 1$.
\end{enumerate}
\end{lemma}

As a less ambitious part of the goal, we first show that, given $\pi$ as in \eqref{sequence} and a fixed $j\in\{1,\dots,m\}$,
we can find a rational function~$h$ such that its only dicritical is $E_j$ (with positive degree), but without further control on that degree.
 In fact, we show simultaneously a similar result for any subset $J\subset \{1,\dots,m\}$.

\begin{prop}\label{prop:one-dic}
For any nonempty $J\subset\{1,\dots,m\}$, there exists a rational function~$h$ for which
$E_j$ is dicritical of some positive degree if $j\in J$, and $E_j$ is not
dicritical if $j\notin J$.
\end{prop}

\begin{proof}
We take $h$ of the form $h:=\prod_{j=1}^m \mathcal{C}_j^{(r_j)}$ where $r_1,\dots,r_m$ are unknowns. Let $N_i:=\sum_{j=1}^m r_j a_{ji}\in \mathbb{Z}$ be the multiplicity of $\pi^*(h)$ along~$E_i$. Then
\[
\begin{pmatrix}
r_1&\dots&r_m
\end{pmatrix}
A=
\begin{pmatrix}
N_1&\dots&N_m
\end{pmatrix}.
\]
Since the matrix~$A$ is unimodular, the choice of $N_j=0$ for $j\in J$ and $N_j\neq 0$ for $j\notin J$
yields a solution for $r_1,\dots,r_m$. For any value of $r_i$
we can choose $\mathcal{C}_i^{(r_i)}$ non constant
and the result follows.
\end{proof}


\section{Construction of special hypersurfaces}
\label{sec:special}

Proposition~\ref{prop:one-dic}
provides no control on the degree of the dicritical component.
This is why we need to introduce another class of
hypersurfaces of~$X_0$
called \emph{special hypersurfaces}.


\begin{lemma}\label{lem:special_hypersurface}
Fix  $s\in\{1,\dots,m\}$ and some $E_j, 1\leq j<s,$  that
intersect $E_s$ positively, i.e., it intersects the general fiber
of the projective space bundle $E_s\subset X_s$, and let $\ell\in\mathbb{Z}_{\geq 1}$. Then there exists a \emph{special hypersurface} $\mathscr{H}_j$ in $X_0$
such that its strict transform $\tilde{\mathscr{H}}_j$ in $X_s$ satisfies:
\begin{enumerate}[label=\rm(\alph{enumi})]
\item $\tilde{\mathscr{H}}_j\cap E_s$ is the union of $E_s\cap E_j$
and possibly fibers of $E_s \to Z_s$,
\item\label{lem:special_hypersurface_b} the intersection multiplicity of $\tilde{\mathscr{H}}_j$ and $E_j$
along $E_s\cap E_j$  equals $1$
(in particular, $\tilde{\mathscr{H}}_j$ is generically smooth at $E_s\cap E_j$),
\item\label{lem:special_hypersurface_c} the intersection multiplicity of $\tilde{\mathscr{H}}_j$ and $E_s$
along $E_s\cap E_j$  equals $\ell$.
\end{enumerate}
\end{lemma}

\begin{remark}
Why are such $\mathscr{H}_j$  needed? In the construction of rational functions~$h$ we will use at least previous hypercurvettes.
 If we want to control the degree
of  $h|_{E_s}$, we may need to compensate the contribution
of such positively intersecting $E_j$ in $E_s$ by the use of these special
hypersurfaces. This will be clear in the proof of Proposition \ref{prop3}.
\end{remark}

\begin{proof}[Proof of Lemma{\rm~\ref{lem:special_hypersurface}}]
Let us consider the blow-up with center~$Z_s$ yielding $X_s$:
\[
\begin{tikzcd}
X_s\ar[r,hookleftarrow]\ar[d,"\pi_s"]&
E_s\ar[d,"\sigma_s"]\ar[r,hookleftarrow]&
\mathbb{P}^{n_s}\\
X_{s-1}\ar[r,hookleftarrow]&Z_s.&
\end{tikzcd}
\]
The hypothesis on the positive intersection implies that $Z_s\subset E_j\subset X_{s-1}$; note that $n_s:=\codim Z_s - 1\geq 1$ and that $Z_s$ is smooth in $X_{s-1}$. Since this ambient space is smooth, the local ring
$R_{Z_s}$ is a regular local ring of dimension~$d:=n_s+1$  with residue field $R(Z_s)$.

Let us fix a Zariski open affine subset $U$ of $X_{s-1}$ such that $Z_s\cap U \neq \emptyset$.
The previous assertion implies that
$R[U]\subset R_{Z_s}$.
Let $f_1,\dots,f_d$ be a regular system of parameters of~$R_{Z_s}$.
Getting rid of eventual denominators, we assume that $f_1,\dots,f_d\in R[U]$. We can moreover choose $f_1$ such that $f_1=0$
is an equation of $E_j$, and we can choose $f_d$ \emph{generic}.

This last element $f_d$ defines an affine chart
$V\subset\pi_s^{-1}(U)$ in $X_s$ such that
\[
R[V]=R[U][w_1,\dots,w_{d}],\qquad\text{where } w_i=\frac{f_i}{f_d}\text{ if }i<d,
\text{ and } w_d=f_d.
\]
The blow-up  $\pi_s:X_s\to X_{s-1}$ restricts to a map $V\to U$ which induces
the above inclusion $R[U]\subset R[V]$. Let $E_{s,j}$ denote the intersection $E_s\cap E_j$.
This inclusion lifts to a map
\[
\begin{tikzcd}[row sep=0pt]
R_{Z_s}\rar& R_{E_{s,j}}:&\\
f_i\rar[mapsto]& w_i w_d&\quad \text{for } i<d,\\
f_d\rar[mapsto]& w_d.&\\
\end{tikzcd}
\]
Note that $R[V]\subset R_{E_{s,j}}$, $\dim R_{E_{s,j}}=2$, and $w_1,w_d$ form a regular system of parameters of $R_{E_{s,j}}$.
For any $\lambda\in\mathbb{C}^*$ the equation $f_1^\ell=\lambda f_d^{\ell+1}$ defines a hypersurface in $U$, whose strict transform
in $V$
is given by $w_1^\ell=\lambda w_d$. By Bertini's argument we may assume that for generic $\lambda$ these hypersurfaces
are irreducible.


We denote by $\tilde{\mathscr{H}}_{s-1,j}$ the Zariski closure in $X_{s-1}$ of the hypersurface given by $f_1^\ell=\lambda f_d^{\ell+1}$ in $U$,
and by $\tilde{\mathscr{H}}_{s,j}$ its strict transform in $X_s$, which is also the Zariski closure of the hypersurface given by $w_1^\ell=\lambda w_d$, in $V$. As a consequence
these closures are irreducible and reduced, as it is the case for
$\mathscr{H}_{j}:=\pi_s(\tilde{\mathscr{H}}_{s,j})\subset X_0$, see~\cite[Chapter~II, Exercices 4.4 and 3.11]{har}.

Note that $\tilde{\mathscr{H}}_{s,j}\cap E_s$ may contain
\emph{vertical} extra components of the form $p^{-1}(B)$, where $B$ is
 a hypersurface in $Z_s$.
\end{proof}

\begin{remark}
More generally, we could replace $\lambda f_d^{\ell+1}$ with a generic $\tilde{f}\in R[U]$, such that $\tilde{f}\in\mathfrak{M}_{Z_s}^{\ell + 1}\setminus\mathfrak{M}_{Z_s}^{\ell + 2}$
and $f_1$ does not divide $\tilde{f}$ in $R_{Z_s}$. Here and in the sequel the notation  $\mathfrak{M}_{C}$ stands for the maximal ideal
of the regular local ring associated to a \emph{point} $C$ in $X_{s-1}$ or $X_s$ (maybe non-closed).
The strict transform in $V$ of the hypersurface in $U$ defined by $f_1^\ell=\tilde{f}$ is then of the form
$w_1^\ell=uw_d $, where $u$ is a unit in $R_{E_{s, j}}$.
\end{remark}

%

\begin{example}\label{ejm:special}
We  construct an example of  a special hypersurface.
Let $\pi:=\pi_2\circ\pi_1$, where $\pi_1$ is the blow-up of $\mathbf{0}\in\mathbb{C}^3$, and $\pi_2$ is the blow-up
along a smooth conic in $E_1\cong\mathbb{P}^2$, say the one with equation $x z - y^2=0$ in homogeneous coordinates.

A special hypersurface for $E_2\cap E_1$ on $E_2$ is obtained using an equation as
\[
0=\mathscr{H}_1(x,y,z):=(x z -y^2)^2 + f_5(x,y,z)+\text{higher degree terms},
\]
where $f_5(x,y,z)$ is a \emph{generic} homogeneous polynomial of degree~5.
The intersection of this hypersurface with $E_1$ is the conic, while the
intersection with $E_2$ is the conic and $10$ generic fibers of the ruled surface $E_2$.
\end{example}

\begin{notation}
If no confusion is expected, we denote for simplicity $\tilde{\mathscr{H}}_{j}$ for both $\tilde{\mathscr{H}}_{s-1,j}$ and $\tilde{\mathscr{H}}_{s,j}$.

\end{notation}

The statement of Lemma~\ref{lem:special_hypersurface} is not enough for our purposes, since
we do not control what happens at non-generic points of $E_{s,j}=E_j\cap E_s$ in $X_s$ and we need to consider \emph{very}
special hypersurfaces.
The algebraic meaning of~\ref{lem:special_hypersurface_b} and~\ref{lem:special_hypersurface_c}
in that lemma is that
\[
\mathfrak{I}_{E_{s,j}}(w_1,\tilde{\mathscr{H}}_j) = \mathfrak{M}_{E_{s,j}}  
\text{ and }
\mathfrak{I}_{E_{s,j}}(w_d, \tilde{\mathscr{H}}_{j}) = \mathfrak{I}_{E_{s,j}}(w_d,w_1^\ell),
\]
respectively, where $\mathfrak{I}_W(\cdot)$ denotes the ideal generated by the elements between brackets in the local ring at a subvariety $W$.
These conditions are generic, but there are (closed and non-closed) points in $E_{s,j}$ which do not satisfy it.
Let
\[
\good_{E_{s,j}}(\tilde{\mathscr{H}}_j):=
\{p \in E_{j, s}\mid
\mathfrak{I}_{p}(w_1,\tilde{\mathscr{H}}_j) = \mathfrak{M}_{p}\text{ and }
\mathfrak{I}_{p}(w_d, \tilde{\mathscr{H}}_{j}) = \mathfrak{I}_{p}(w_d,w_1^\ell)
\}
\]
denote the set of points (closed or not) in $E_{s,j}$ which do satisfy the similar conditions in their local ring.

\begin{lemma} Let $s$, $E_j$ and $\ell$ be as in Lemma{\rm~\ref{lem:special_hypersurface}}.
Let $p_1,\dots,p_r\in E_{s,j}$
and let $q_1,\dots,q_k\in Z_s$; they may be closed or not. Then there is a special hypersurface $\mathscr{H}_j$ as in Lemma{\rm~\ref{lem:special_hypersurface}}
such that $p_1,\dots,p_r,\sigma_s^*(q_1),\dots,\sigma_s^*(q_k)\in \good_{E_{s,j}}(\tilde{\mathscr{H}}_j)$.
\end{lemma}

\begin{proof}
We will refine the proof of Lemma~\ref{lem:special_hypersurface}. Let us denote  $\hat{p_i}:=\pi_s(p_i)$ for $i=1,\dots,r$. Since $X_{s-1}$ is quasi-projective,
we can choose an affine chart~$U$ of $X_{s-1}$
such that
$\hat{p}_1,\dots,\hat{p}_r,q_1,\dots,q_k\in U$.

Let $\hat{p}$ be any one of these points. We have the inclusions
\[
R[U]\subset R_{\hat{p}}\subset R_{Z_s}\subset R_{E_j}\subset R(X_{s-1}) = R(X_0).
\]
The affine chart $V$ of $X_s$ depends on the choice of $f_d$. For any such choice $\sigma_s^*(q_1),\dots,\sigma_s^*(q_k)\in V$,
and for generic choices we have that $p_1,\dots,p_r\in V$. Let us fix  one of
these $k+r$ points, denote it as $\tilde{p}$, and let $\hat{p}\in Z_s$ be the corresponding point.

Let $f_{\hat{p}}$  be a generator of $\mathfrak{I}_{\hat{p}}(E_j)$ (it replaces $f_1$ in the proof of Lemma~\ref{lem:special_hypersurface})
which belongs to $R[U]$; in particular it is also a generator of  $\mathfrak{I}_{Z_s}(E_j)$.
The special hypersurface $\tilde{\mathscr{H}}_{j,\hat{p}}\subset X_{s-1}$ is defined in $U$ by $f_{\hat{p}}^{\ell} - \mu_{\hat{p}} f_d^{\ell+1}$, with
$\mu_{\hat{p}}$  generic in $\mathbb{C}^*$,
and its strict transform $\tilde{\mathscr{H}}_{j,\tilde{p}}\subset X_s$ is defined in $V$ by  $F_{\tilde{p}}:=w_{\tilde{p}}^{\ell} - \mu_{\hat{p}} w_d$, where
$w_{\tilde{p}}:=\frac{f_{\hat{p}}}{f_d}$ and $w_d:=f_d$.

We have another chain of inclusions
\[
R[V]\subset R_{\tilde{p}}\subset R_{E_{s,j}}\subset R_{{\tilde{\mathscr{H}}_{j,\tilde{p}}}}\subset R(X_{s}) = R(X_0).
\]
The above choices guarantee that $w_{\tilde{p}}$ is in $R[V]$ and is a generator of $\mathfrak{I}_{\tilde{p}}(E_j)$, and also a generator
of $\mathfrak{I}_{E_{s,j}}(E_j)$. The goal of this construction was to guarantee that $\tilde{p}\in\good_{E_{s,j}}(\tilde{\mathscr{H}}_{j,\tilde{p}})$. Indeed,
the coefficients of $w_{\tilde{p}}^{\ell}$ and $w_d$ in $F_{\tilde{p}}$ are units in $R_{\tilde{p}}$.

Let $\vecl\in\mathbb{C}^{r+k}$ with coordinates $\lambda_{\tilde{p}}$.
For a fixed $\tilde{p}_0$, we have for all $\tilde{p}$ that $w_{\tilde{p}}=u_{\tilde{p},\tilde{p}_0}w_{\tilde{p}_0}$,
where $u_{\tilde{p},\tilde{p}_0}$ is a unit in $R_{E_{s,j}}$, and $u_{\tilde{p}_0,\tilde{p}_0}=1$.
Let
\[
F_\vecl:=\sum_{\tilde{p}} \lambda_{\tilde{p}} F_{\tilde{p}}=
\left(\sum_{\tilde{p}} \lambda_{\tilde{p}} u_{\tilde{p},\tilde{p}_0}^\ell\right) w_{\tilde{p}_0}^\ell -
\left(\sum_{\tilde{p}} \lambda_{\tilde{p}} \mu_{\tilde{p}}\right) w_d.
\]
This function of $R[V]$ defines a special hypersurface $\tilde{\mathscr{H}}_{j,\vecl}\subset X_s$.
For generic values of $\vecl$, we have that $\tilde{p}_0\in\good_{E_{s,j}}(\tilde{\mathscr{H}}_{j,\vecl})$, since the conditions on the coefficients
are fulfilled. Hence, we can choose $\vecl$ such that  $\tilde{p}\in\good_{E_{s,j}}(\tilde{\mathscr{H}}_{j,\vecl})$ for all~$\tilde{p}$
and the desired special hypersurface is constructed.
\end{proof}

\begin{remark}\label{choice points}
The application of this lemma will be the following one. We will take as points $p_1,\dots,p_r\in E_{s,j}$ the
non-empty intersections of the images
of \lq later\rq\ centers of blow-ups with $E_{s,j}$, and as points $q_1,\dots,q_k\in Z_s$ the non-empty intersections of the images
of such centers of blow-ups with $Z_s$.
\end{remark}


\section{Functions with prescribed dicritical conditions}
\label{sec:one_dic}

This section contains the technical results that refine
Proposition~\ref{prop:one-dic} in order to reach Goal~\ref{goal}.  First, in Proposition \ref{prop3} below,
we essentially prove the existence of a rational function for which
the only dicritical component is the last one, for any prescribed
degree.

Note that, in order to realize the full goal, it is not enough to apply Proposition \ref{prop3} for each partial
$\pi_{s,1}$ (for which $E_s$ is the last component),
since in that proposition there is a priori no control on what happens
with the components $E_{s+1},\dots,E_m$.  For that we will need a more involved strategy, resulting in the more general Theorem \ref{thm4}.

\begin{prop}\label{prop3}
Let $s\in\{1,\dots,m\}$ and $d\geq 1$.
There exists  a rational function $h$ for which $E_i,i<s,$ is not dicritical,
and $E_s$ is dicritical of degree~$d$.
\end{prop}
\begin{proof}
We can assume $s=m$, since we do not deal with $E_{s+1},\dots,E_m$.
We start with some further notation.
We break the set of components $\{E_1,\dots,E_{s-1}\}$ into two classes. A divisor $E_i$
is in the first class if and only if it does not contain the center which produces $E_s$.

For the sake of simplicity, we assume that the first class contains exactly the first divisors
$E_1,\dots,E_{s-k-1}$; this is not necessarily true but notations are simpler and it
does not affect the arguments in any way.
Hence we will be mainly interested in what happens with $E_{s-k},\dots,E_{s-1}$.
Since we will need the special hypersurfaces of Lemma~\ref{lem:special_hypersurface},
we choose $\ell_j\in\mathbb{Z}_{>0}$ for $s-k\leq j < s$.

Actually this proof
will work with the choice $\ell_j=1$, but we admit some freedom of choice since for other results
we may need to take large $\ell_j$'s.

We need more choices. For each $j\in\{1,\dots,s-1\}$,
we choose generic $E_j$-hypercurvettes $\mathcal{C}_{j,i}$, for some $i$'s. Next, we choose two generic $E_s$-hypercurvettes $\mathcal{C}'_s,\mathcal{C}''_s$.
So all the varieties  $\mathcal{C}_{j,i}$, all $E_i \  (1\leq i \leq s)$ and $\mathcal{C}'_s,\mathcal{C}''_s$ form a normal crossing divisor in $X_s$.
Finally, for $j\in\{s-k,\dots,s-1\}$, let us consider a special hypersurface
$\mathscr{H}_{j}$ as in
Lemma~\ref{lem:special_hypersurface}, still with the normal crossing restriction.

%
%

We will construct a rational function of the form
\begin{equation}\label{eq:candidate}
h:=\frac{(\mathcal{C}_s')^d}{(\mathcal{C}_s'')^d}\dfrac{\displaystyle\prod_{i=1}^{s-1}\mathcal{C}_i^{(r_i)}}
{\displaystyle\prod_{j=s-k}^{s-1}\mathscr{H}_j^{r'_j}},
\end{equation}
for some integers $r_i,r'_j$ to be determined. Here the symbols $\mathcal{C}_{i}^{(r_i)}$ come from Notation~\ref{notation:ri}, but the $\mathscr{H}_j^{r'_j}$ are \emph{honest} powers of the $\mathscr{H}_j$.
We will see later that the $r'_j$ can be chosen positive.
As usual, the multiplicities $N_i$ are given by
\begin{equation*}
\divi\pi^*h=(\text{strict transform of} \divi h) + \sum_{i=1}^s N_i E_i.
\end{equation*}
Following Lemma~\ref{lem:numer_dic}, our first goal is to provide conditions such that $N_s=0$ and $N_i\neq 0$ if $i<s$.
Note that there is no numerical contribution of $(\mathcal{C}_s')^d\cdot(\mathcal{C}_s'')^{-d}$ to the~$N_i$.
Let $h_s:=(\pi^*h)_{|E_s}:E_s\dashrightarrow\mathbb{C}$.
Note that
\[
\divi h_s=d(\mathcal{C}_s'\cap E_s) - d(\mathcal{C}_s''\cap E_s) + \sum_{j=s-k}^{s-1}(N_j -r'_j\ell_j)(E_j\cap S) +\sigma_s^*B,
\]
for some divisor $B$ on $Z_s$.
Applying Lemma~\ref{lem:numer_dic}\ref{lem:numer_dic3}, the presence of $(\mathcal{C}_s')^d$ and $(\mathcal{C}_s'')^d$
implies that $E_s$ is dicritical
for~$h$ of degree~$d$ if $N_j=r'_j\ell_j$ for $s-k\leq j\leq s-1$.
We will show that there exist $r_i$ and $r'_j$ such that these last $k$ equalities are satisfied and $N_i \neq 0$ for all $i<s$.

We recall that the unimodular matrix $A\in\mat(s;\mathbb{Z})$ is defined by \eqref{eq:matrix}.
We define another matrix $B\in\mat(k\times s;\mathbb{Z})$ with coefficients $b_{ij}$ satisfying
\[
\pi^* \mathscr{H}_{s-k+i-1}=(\text{strict transform of} \mathscr{H}_{s-k+i-1}) + \sum_{j=1}^s b_{ij} E_j
\]
for $i=1,\dots,k$. The relations between all those numbers come from the matrix identity
\begin{gather}\label{eq:matricial}
\begin{pmatrix}
r_1&\dots& r_{s-1}&0&-r'_{s-k}&\dots&-r'_{s-1}
\end{pmatrix}
\begin{pmatrix}
A\\
B
\end{pmatrix}=\\
\nonumber
\begin{pmatrix}
N_1&\dots& N_{s-k-1}&r'_{s-k}\ell_{s-k}&\dots&r'_{s-1}\ell_{s-1}&0
\end{pmatrix}.
\end{gather}
This matrix equation consists of $s+k$ linear equations. Let us make explicit
the last $k+1$ equations, denoted as $e_{s-k},\dots,e_{s-1},e_{s}$. We first recall some properties of $A$,
see Example~\ref{ejm:matrix}:
\begin{equation}\label{eq:rels_A}
a_{is}=\sum_{j=s-k}^{s-1} a_{ij},\quad 1\leq i<s,\quad\text{ and } \quad
a_{ss}=\sum_{j=s-k}^{s-1} a_{sj} +1.
\end{equation}
Using the same ideas, we have
\begin{equation}\label{eq:rels_B}
b_{is}=\sum_{j=s-k}^{s-1} b_{ij}+\ell_{s-k+i-1},\quad 1\leq i\leq k.
\end{equation}
The equations $e_j$, for $s-k  \leq j \leq s-1,$ are
\[
(e_j):\quad\sum_{i=1}^{s-1} r_i a_{ij} - \sum_{i=1}^k r'_{s-k+i-1} b_{ij}= r'_{j}\ell_{j}.
\]
Applying \eqref{eq:rels_A} and \eqref{eq:rels_B} we obtain
\[
\sum_{j=s-k}^{s-1}(e_j):\quad
\sum_{i=1}^{s-1} r_i a_{is} - \sum_{i=1}^k r'_{s-k+i-1} (b_{is}-\ell_{s-k+i-1})=\sum_{j=s-k}^{s-1}  r'_{j}\ell_{j},
\]
which is equivalent to
\[
\sum_{j=s-k}^{s-1}(e_j):\quad
\sum_{i=1}^{s-1} r_i a_{is} - \sum_{i=1}^k r'_{s-k+i-1} b_{is}=0,
\]
i.e, equation $e_{s}$.
Since the last equation is a consequence of the previous ones, we can eliminate it.
Denoting
\[
C:=
\begin{pmatrix}
0&\dots&0&\ell_{s-k}&\dots&0\\
\vdots&\ddots&\vdots&\vdots&\ddots&\vdots\\
0&\dots&0&0&\dots&\ell_{s-1}
\end{pmatrix}\in\mat(k\times {(s - 1)};\mathbb{Z}),
\]
the matrix identity \eqref{eq:matricial}, without its last equation, can be rewritten as
\begin{gather}
\begin{pmatrix}
r_1&\dots& r_{s-1}
\end{pmatrix}
A_{s-1}- \begin{pmatrix} r'_{s-k}&\dots&r'_{s-1}  \end{pmatrix}   B_{k,s-1}      =         \\
\nonumber
\begin{pmatrix}
N_1&\dots& N_{s-k-1}&0&\dots&0
\end{pmatrix}
+
\begin{pmatrix}
0&\dots&0&r'_{s-k}\ell_{s-k}&\dots&r'_{s-1}\ell_{s-1}
\end{pmatrix},
\end{gather}
where $A_{s-1}$ is the principal submatrix of $A$ formed by its first $s-1$ rows and columns, and $B_{k,s-1}$ is obtained from the first $s-1$ columns of $B$.
This is equivalent to
\begin{gather}
\begin{pmatrix}
r_1&\dots& r_{s-1}
\end{pmatrix}
A_{s-1}=\\
\nonumber
\begin{pmatrix}
N_1&\dots& N_{s-k-1}&0&\dots&0
\end{pmatrix}
+
\begin{pmatrix}
r'_{s-k}&\dots&r'_{s-1}
\end{pmatrix}
(B_{k,s-1} + C).
\end{gather}
 Since the matrix $A_{s-1}$ is unimodular, if we fix arbitrary values for $N_1,\dots,N_{s-k-1}$ and $r'_{s-k},\dots,r_{s-1}'$,
the system has a solution for $r_1,\dots,r_{s-1}$. We just have to choose all these values nonzero.
%
\end{proof}

\begin{example}\label{example_special}
 We construct $\pi$ as the composition of the following three point blow-ups.
As usually,  $\pi_1$ is the blow-up of $Z_1=0\in\mathbb{C}^3= X_0$. We use coordinates $x,y,z$ in $\mathbb{C}^3$, as well as in the various charts of $X_i, i>0$, in the usual way. We take $Z_2$ as  the origin  of the chart in $X_1$ where $E_1$ is given by $z=0$, and $Z_3$ as the origin  of the chart in $X_2$ where $E_2$ is given by $y=0$.

The following can be verified by explicit computations.  We can take a class of hypercurvettes $\mathcal{C}_1$ given by $ax+by+cz$, for generic choices of the coefficients.
Similarly, we can take a class of hypercurvettes $\mathcal{C}_2$ of the form $cx+dy$, for generic choices of the coefficients.  For $\mathcal{C}_3'$ and $\mathcal{C}_3''$ we can take $x+z^2$ and $x-z^2$, respectively.

As special hypersurfaces $\mathscr{H}_1$ and $\mathscr{H}_2$ for $E_3$, associated to the intersections $E_1\cap E_3$ and $E_1\cap E_3$, we can take  $x^2+z^2y$ and $x^2z+y^3$, respectively, where we took $\ell_1=\ell_2=1$ for simplicity.

Concentrating on the base case $d=1$, we thus look for a rational function $h$  of the form
\[
\frac{\mathcal{C}_3'}{\mathcal{C}_3''}
\frac{\mathcal{C}_1^{(r_1)} \mathcal{C}_2^{(r_2)} }
{\mathscr{H}_1^{r'_1} \mathscr{H}_2^{r'_2}}.
\]
With the above choices one checks that, with the notation of the proof above,
\begin{gather}\label{eq:matrixexample}
\begin{pmatrix}
A\\
B
\end{pmatrix}=
\nonumber
\begin{pmatrix}
1 & 1 & 2 \\
1 & 2 & 3 \\
1 & 2 & 4 \\\hline
2 & 4 & 7 \\
3 & 5 & 9
\end{pmatrix}.
\end{gather}
The  identity \eqref{eq:matricial}  becomes
\begin{gather}
\begin{pmatrix}
r_1& r_2&0&-r'_{1}&-r'_{2}
\end{pmatrix}
\begin{pmatrix}
A\\
B
\end{pmatrix}=\\
\nonumber
\begin{pmatrix}
r'_{1}&r'_{2}&0
\end{pmatrix},
\end{gather}
which is equivalent to the equalities
\begin{equation}\label{r-equations}
 r_1=2r'_1\qquad\text{ and }\qquad r_2=r'_1+3r'_2.
\end{equation}
Also, independent of choices, in the chart of $X_3$ where $E_1$, $E_2$ and $E_3$ are given by $z=0$, $y=0$ and $x=0$, respectively, one can calculate that the restriction of $\pi^*h$ to $E_3$ is  $\frac{1+z}{1-z}$.  So indeed $E_3$ is a dicritical of degree $1$ for $h$.
\end{example}

\begin{remarks}\label{remark positive}
\mbox{}

\begin{enumerate}[label=(\arabic{enumi})]
\item
In the proof of Proposition~\ref{prop3}, we chose first arbitrary positive $\ell_j$, for $s-k\leq j \leq s-1$. And then the $N_j$ and $r'_j$ for $s-k\leq j\leq s-1$ had to be related as $N_j=r'_j\ell_j$.

 In the proof of Theorem \ref{thm4} below, we will start by fixing positive $r'_j$, for $s-k\leq j\leq s-1$. Then we will choose the $\ell_j$ \lq big enough\rq\ with respect to the $r'_j$ and the geometry of the modification, in a sense that we will be made precise. Note that in such a strategy, the chosen special hypersurfaces $\mathscr{H}_{j}$ depend also on the $r'_j$.
 Also the $N'_j \, (1\leq j <s-k)$ will be chosen \lq big enough\rq\ with respect to the $r'_j$.

\item In fact, in the proofs of both Proposition \ref{prop3} and Theorem \ref{thm4}, we could simply take $r'_{s-k}=\dots=r'_{s-1}=1$. But the arguments are really the same for more general $r'_j$.
\end{enumerate}

\end{remarks}

A first approach to solve the general problem would be to prove that, if we perform
more blow-ups, the function still has $E_s$ as  only dicritical.
We will see in Example~\ref{ejm:special_ok} that, with the procedure of Proposition~\ref{prop3},
it may  not be possible to find such a function. The proof of our main theorem below requires another strategy. We start with
a rational function provided by Proposition~\ref{prop3} and we modify it to achieve the goal.



\begin{thm}\label{thm4}
Let $\pi$ be a sequence of $m$ blow-ups as in \eqref{sequence}, let $s\in\{1,\dots,m\}$, and $d\geq  1$.
Then there exists a rational function $h$ for which $E_i$ is not dicritical for $i\neq s$,
and $E_s$ is dicritical of degree~$d$.
\end{thm}

The (long) proof of this theorem is moved to the last section in order to do not break the sequence of ideas.
However, Example \ref{example_more special} below already exhibits some aspects of that proof.

\begin{example}\label{ejm:special_ok}(continuing Example~\ref{ejm:special})
We look for a function $h$ for which $E_1$ is a dicritical of multiplicity~$1$ and $E_2$ is not dicritical.

Of course, the naive method of using $h:=\frac{\mathcal{C}_1'}{\mathcal{C}_1''}$, for $\mathcal{C}_1',\mathcal{C}_1''$
generic hypercurvettes of $E_1$ (e.g, generic linear functions) does not work, since $h$ is not constant
on~$E_2$. Choose another generic hypercurvette $\mathcal{C}_1$ and consider
\[
h(x.y,z)=\frac{\mathcal{C}_1'(x z - y^2)^k}{\mathcal{C}_1''(x z - y^2)^k + \mathcal{C}_1^\grande}.
\]
It is not difficult to check that, if $2k+1<\grande$, then $E_1$ is dicritical for $h$ (of degree~$1$).
Moreover, if $\grande<3k+1$, then $h$ is constant on $E_2$.
\end{example}

\begin{example}(continuing Example \ref{example_special})\label{example_more special}
We add a blow-up $\pi_4$, namely the one with center the curve $Z_4$, given by $x=y+1=0$ in the chart of $X_3$ where $E_3$ is given by $x=0$.
One verifies that $xz+y^2$ can be taken as hypercurvette $\mathcal{C}_4$.

We fix some of the free parameters in the construction of the rational function~$h$
in Example~\ref{example_special}: we choose $r'_1=r'_2=1$, implying by \eqref{r-equations} that $r_1=2$ and $r_2=4$, and we take $\mathcal{C}_1^{(2)}=(x+y+z)^2$ and $\mathcal{C}_2^{(4)}=(x+y)^4$.
We can check that $h$ is also dicritical for $E_4$, so it is not a valid function for the
statement of Theorem~\ref{thm4}.

Let us try to fix it.
The expression \eqref{general h} becomes $h=\frac{f}{g}$ where
\[
f=(x+z^2)(x+y+z)^2(x+y)^4 \quad\text{ and }\quad g=(x-z^2)(x^2+z^2y)(x^2z+y^3).
\]
Let us consider an extra hypercurvette $\mathcal{C}_3 = 2x+z^2$ for $E_3$
and define
\[
h' = \frac{(x+z^2)(x+y+z)^2(x+y)^4 (xz+y^2)^k}{(x-z^2)(x^2+z^2y)(x^2z+y^3)(xz+y^2)^k +(2x+z^2)^\ell},
\]
where $k,l\in\mathbb{Z}_{>0}$ will be determined. It is not hard to see that $h$ and $h'$
have the same behavior in $E_1,E_2$. Moreover, if
\[
5+\frac{3}{2} k < \ell
\]
then $h'$ is also a dicritical for $E_3$ with degree~$1$
Finally, a condition to be constant for $E_4$ is given by
\[
\ell < 5 + \frac{7}{4} k.
\]
In the proof of Theorem~\ref{thm4}, we will see several conditions of the same nature in \eqref{eq:condition}.
With a choice of $k\gg 0$, e.g. $k=5$, there are choices for $\ell$, e.g. $\ell=13$.
\end{example}

With the previous results we can finally prove the main theorem.
Let $\pi$ be a sequence of $m$ blow-ups as in \eqref{sequence}.
Since the set of values of non-dicritical components is finite, the following result is straightforward.

\begin{lemma}\label{lemma6} Fix $s\in\{1,\dots,m\}$.
Let $h$ a rational function for which $E_s$ is dicritical of degree~$d_s\geq 0$ and the $E_i, i\neq s,$ are not dicritical.
Then, for generic $a,b\in\mathbb{C}$, the function $g:=\frac{h-a}{h-b}$ satisfies the same condition
and $\pi^*(g)(E_i)\in\mathbb{C}^*$, if $i\neq s$.
\end{lemma}

\begin{prop}\label{prop7}
Let $\emptyset \neq J\subset\{1,\dots,m\}$. Assume that for each $j\in J$ there is a rational function $h_j$ for which $E_j$ is a dicritical of
degree $d_j\geq 0$ and $E_i, i\neq j,$ is not dicritical.
Then there is a rational function $h$ for which $E_j$ is dicritical of degree~$d_j$ if $j\in J$ and $E_j$
is not dicritical if $j\notin J$.
\end{prop}

\begin{proof}
It is enough to choose generic $a_j,b_j$ as in Lemma~\ref{lemma6} and consider
\[
h:=\prod_{j\in J}\frac{h_j-a_j}{h_j-b_j}.
\qedhere
\]
\end{proof}

Finally, combining  Theorem~\ref{thm4} and Proposition~\ref{prop7}, we end the proof
of Theorem~\ref{thm:main}, likely reaching an optimal result in our setting.



\begin{remark}\label{rmk_product}
Let us assume that $E_i$ is dicritical for two functions $h_1,h_2$, with degrees $0\leq d_1\leq d_2$.
Let $h:=h_1 h_2$.
\begin{enumerate}[label=(\arabic{enumi})]
\item In general $E_i$ will be a dicritical of degree~$d$ for $h$, with $d_2-d_1\leq d\leq d_1+d_2$; if $d_1=d_2$,
then the option to be non-dicritical may also happen.

\item If $d_2>d_1=0$, then $E_i$ is dicritical of degree~$d_2$ for $h$.
\end{enumerate}
\end{remark}

\section{Proof of Theorem~\ref{thm4}}\label{sec:proof_thm}

Let us consider first a rational function $h'$ of the form \eqref{eq:candidate},  constructed as in the proof of Proposition~\ref{prop3}, and additionally taking into account Remarks ~\ref{choice points} and ~\ref{remark positive}.
This function $h'$ thus depends on some positive integer numbers
\[
r'_{s-k},\dots,r'_{s-1}, \ell_{s-k},\dots,\ell_{s-1}, N'_{1},\dots,N'_{s-k-1},
\]
where we assume that $N'_{1},\dots,N'_{s-k-1},\ell_{s-k},\dots,\ell_{s-1}$ are big enough, in a sense that will be made precise in Step \ref{general positivity}.
We put also $N'_i:=r'_i \ell_i$, for $s-k\leq i\leq s-1$, hence, we have data for $N'_1, \dots, N_{s-1}'$.

We can write $h'$ as $h'=\frac{f}{g}$, where
$f, g\in R[X_0]$.
Note that $f$ is a product of hypercurvette-polynomials and $g$ is a product of hypercurvette-polynomials and defining polynomials of the special hypersurfaces~$\mathscr{H}_{j}$.
With the above notation,
\[
\divi\pi^*h'=(\text{strict transform of} \divi h')+\sum_{i=1}^m N'_i E_i,\qquad
N'_i=\nu_i(h').
\]
For $j>s$, we choose a generic hypercurvette $\mathcal{C}_j$ of $E_j$ (in particular, we assume normal crossing behavior). Take also another generic hypercurvette~$\mathcal{C}_s$
for~$E_s$. Our candidate for the rational function $h$ is
\begin{equation}\label{general h}
h:=\frac{f\displaystyle\prod_{j=s+1}^m \mathcal{C}_j^{k_j}}{g\displaystyle\prod_{j=s+1}^m \mathcal{C}_j^{k_j}+\mathcal{C}_s^\grande}
\text{  for some }\grande\in\mathbb{Z}_{>0}\text{ and }\mathbf{k}:=(k_{s+1},\dots,k_m)\in\mathbb{Z}_{>0}^{m-s}.
\end{equation}
Here the expressions $\mathcal{C}_s^{\grande}$ and $\mathcal{C}_j^{k_j}$ are \emph{honest} powers.
As for $h'$, we have
\[
\divi\pi^*h=(\text{strict transform of} \divi h)+\sum_{i=1}^m N_i E_i,\qquad N_i=\nu_i(h).
\]
The goal is to find $\mathbf{k},\grande$ for which $h$ satisfies the properties required in the assertion of the theorem.
For any $i=1,\dots,m$ we set
\begin{equation}\label{eq:alphabeta}
\alpha_i:=\nu_i(f)+\sum_{j=s+1}^m k_j\nu_i(\mathcal{C}_j)\quad\text{ and } \quad
\beta_i:=\nu_i(g)+\sum_{j=s+1}^m k_j\nu_i(\mathcal{C}_j).
\end{equation}
Since $\mathcal{C}_s$ is a generic hypercurvette,
\begin{equation}\label{eq:min}
N_i = \nu_i(h) = \alpha_i-
\min\left(\beta_i,\grande\nu_i(\mathcal{C}_s)\right)= \alpha_i-
\min\left(\beta_i,\grande a_{si}\right),
\end{equation}
and $N'_i = \nu_i(h') = \alpha_i-\beta_i\leq N_i$.
Recall that $\nu_i(\mathcal{C}_j)=a_{ji}$.

The statement follows if we prove that there are choices for the parameters
such that $N_i>0$ if $i\neq s$, $N_s=0$,
and $E_s$ is dicritical of degree~$1$ for~$h$.

\begin{case}\label{case1}
For $1\leq i<s$, we have that $N_i>0$. In particular, for any value of $\mathbf{k},\grande$
such $E_i$ is not dicritical.
\end{case}

\begin{proof}[Proof of Case{\rm~\ref{case1}}]
We have $N_i\geq N_i'>0$.
\end{proof}

\begin{case}\label{case2}
If $\frac{\alpha_s}{a_{ss}}<\grande$, then the divisor $E_s$ is dicritical of degree~$d$ for $h$, in particular we have that $N_s=0$.
\end{case}

\begin{proof}[Proof of Case{\rm~\ref{case2}}]
Since $N'_s=0$, we have $\alpha_s=\beta_s$.
Then, if $\frac{\alpha_s}{a_{ss}}<\grande$,
we have that $\beta_s=\alpha_s<\grande a_{ss}$. Then $N_s=N'_s=0$.
Moreover, $(\pi^* h)_{E_s}=(\pi^* h')_{E_s}$; hence  $E_s$ is dicritical of degree $d$ for $h$, since it is for $h'$.
\end{proof}

The rest of the proof is devoted to the main case $i>s$. For a fixed $i>s$ we introduce the notation
\begin{align*}
\zz_i\!:=& \{j\!\in\!\{\!i,\dots,m\}\!\mid\! Z_i\!\subset\! \pi_{i-1,j}(\mathcal{C}_j)\},   \\  
\ct_i\!:=&\{j\!\in\!\{\!1,\dots,i-1\}\!\mid\! Z_i\subset E_j\},\\
\spl_i\!:=&\{j\!\in\!\{\!s\!-\!k,\dots,s\!-\!1\}\!\mid\! Z_i\subset\tilde{\mathscr{H}}_j\},  
\end{align*}
always considered in $X_{i-1}$.
Note that $\zz_i \neq \emptyset$ since $i\in\zz_i$.  Also $\ct_i \neq \emptyset$, since by construction $Z_i$ must be contained in at least one exceptional component.

We start with some preliminary computations.

\begin{step}\label{casoalpha}
For $j\in\zz_i$, let $\mu_i(\mathcal{C}_j)$ be the multiplicity of $\mathcal{C}_j$ along $Z_i$  (in $X_{i-1}$). Then
\[
\alpha_i=
\sum_{a\in\ct_i}\alpha_a + \sum_{j\in\zz_i} k_j \mu_{i}(C_j).
\]
\end{step}

\begin{proof}[Proof of Step{\rm~\ref{casoalpha}}]
We compute separately $\nu_i(f)$ and $\nu_i(\mathcal{C}_j)$ using Lemma~\ref{lem:inductive}. Note that the strict transform of $f$
comes from \lq previous\rq\ (generic) hypercurvettes, so its multiplicity along $Z_i$ vanishes.
This is also the case for $\mathcal{C}_j$ if $j\notin\zz_i$; hence
\[
\nu_i(f) = \sum_{a\in\ct_i} \nu_a(f),\quad
\nu_i(\mathcal{C}_j) =
\begin{cases}
\displaystyle\sum_{a\in\ct_i} \nu_a(\mathcal{C}_j) & \text{ if } j\in \{i,\dots,m\} \setminus \zz_i,\\
\displaystyle\sum_{a\in\ct_i} \nu_a(\mathcal{C}_j) + \mu_i(\mathcal{C}_j)& \text{ if }j\in\zz_i.
\end{cases}
\]
From the definition of the $\alpha$-coefficients in \eqref{eq:alphabeta} we obtain the statement.
\end{proof}

\begin{step}\label{casobeta}
For $j\in\spl_i$, let analogously $\mu_i(\tilde{\mathscr{H}}_j)$ be the multiplicity of $\tilde{\mathscr{H}}_j$ along $Z_i$ (in $X_{i-1}$). Then
\[
\beta_i=
\sum_{a\in\ct_i}\beta_a + \sum_{j\in\zz_i} k_j \mu_{i}(C_j)+ \sum_{j\in\spl_i} r'_j \mu_i(\tilde{\mathscr{H}}_j ).
\]
\end{step}

\begin{proof}[Proof of Step{\rm~\ref{casobeta}}]
The only difference with the proof of Step~\ref{casoalpha} is that we have to add the contribution
of $\nu_i(\tilde{\mathscr{H}}_j)$, when $j\in\spl_i$.
\end{proof}

Note  that $\tilde{\mathscr{H}}_j$ is not necessarily smooth in $X_i, i>s$. For example, take $n=3$ and consider a local system of parameters $x,y,z$ such that $E_s$, $E_j$ and $\tilde{\mathscr{H}}_j$  are given by $z=0$, $y=0$ and $z=y^\ell$, respectively, in $V\subset X_s$.  When moreover $Z_{s+1}$ is given by $x=z=0$, then in some chart of $X_{s+1}$ we have that $\tilde{\mathscr{H}}_j$ is given by $zx=y^\ell$ and hence it is singular.


\begin{step}\label{caseNprime}
For $i>s$, we have that
\begin{equation}\label{formula N'_i}
N'_i=\sum_{a\in\ct_i} N'_a - \sum_{j\in\spl_i} r'_j \mu_i(\tilde{\mathscr{H}}_j).
\end{equation}
\end{step}

\begin{proof}[Proof of Step{\rm~\ref{caseNprime}}]
The formula is obtained by subtracting the formulas in Steps~\ref{casoalpha} and~\ref{casobeta}.
\end{proof}


\begin{step}\label{imageofZ_i}
Fix $i>s$.

\begin{enumerate}[label=\rm(\arabic{enumi})]
\item\label{step4-1} There exists $a\leq s$ such that  $\pi_{s,i-1}(Z_i) \subset E_a$. 

\item\label{step4-2} If $\spl_i \neq \emptyset$, then there exists $a<s$ such that $\pi_{s,i-1}(Z_i) \subset E_a$.

\end{enumerate}

\end{step}

\begin{proof}[Proof of Step{\rm~\ref{imageofZ_i}}]
Let us start with
\ref{step4-1}. In any case, $Z_i \subset E_q$ for some $q<i$, and then $\pi_{s,i-1}(Z_i) \subset \pi_{s,i-1}(E_q)$.
Either $q\leq s$, and then $\pi_{s,i-1}(Z_i) \subset E_q$, or $q >s$, and then $\pi_{s,i-1}(Z_i) \subset \pi_{s,q-1}(Z_q)$.
In this last case, there exist analogously $p<q \  (<i)$ such that $Z_q \subset E_p$, and we have two possible similar conclusions for $\pi_{s,q-1}(E_p)$.
Continuing this way, we obtain on $X_s$ a priori the two following possibilities.
Either
\begin{enumerate}[label=\rm(\alph{enumi})]
  \item  $\pi_{s,i-1}(Z_i) \subset E_a$ for some $a\leq s$, or
 \item for all $a\leq s$ we have that $\pi_{s,i-1}(Z_i) \not\subset E_a$, and $\pi_{s,i-1}(Z_i) \subset \pi_{s,s}(Z_{s+1}) = Z_{s+1}$.
\end{enumerate}
But this last case is not possible since $Z_{s+1}$ must be contained in some $E_a, a\leq s$.
\smallskip

For \ref{step4-2}, let $Z_i \subset \tilde{\mathscr{H}}_j$, hence $\pi_{s,i-1}(Z_i) \subset \tilde{\mathscr{H}}_j$ in $X_s$. Note now that, by construction of $\tilde{\mathscr{H}}_j$, we have that $(E_s \cap \tilde{\mathscr{H}}_j)|_V = (E_s \cap E_j)|_V$.
Since, by construction of $V$, it contains $\pi_{s,i-1}(Z_i)$, we have either $\pi_{s,i-1}(Z_i) \subset E_k$ for some $k\neq j, k\neq s$, or $\pi_{s,i-1}(Z_i) \subset E_s\cap E_j \  (\subset E_j)$.
\end{proof}

\begin{step}\label{general positivity}
 Given $r'_{s-k}, \dots, r'_{s-1}$, we can choose $N'_1, \dots, N'_{s-k-1}$ and $\ell_{s-k}, \dots, \ell_{s-1}$ big enough, depending on the $r'_j$ and the geometry of the last part $\pi_m \circ \dots \circ \pi_{s+1}$ of the given modification $\pi$, such that $N'_i \geq 0$ for all $i>s$. More precisely, we have the following:

\begin{enumerate}[label=\rm(\arabic{enumi})]
\item\label{step5-1} If there exists $b<s$ such that $\pi_{s,i-1}(Z_i) \subset E_b$, in particular if $\spl_i \neq \emptyset$, then $N'_i >0$.

\item\label{step5-2} If  $\displaystyle\pi_{s,i-1}(Z_i) \not\subset \bigcup_{b<s}E_b$, then $N'_i = 0$.
 \end{enumerate}
\end{step}

\begin{proof}[Proof of Step{\rm~\ref{general positivity}}]
For \ref{step5-1}, note that, if $\pi_{s,i-1}(Z_i) \subset E_b$, then $Z_i \subset \pi_{s,i-1}^{-1}(E_b)$. Hence, in the expression \eqref{formula N'_i} for $N'_i$, at least one of the $N'_a$ will contain $N'_b$ as summand.  The idea is to take all $N'_b, b<s,$ big enough, in order to compensate for the negative contributions in  $\eqref{formula N'_i}$.

We present a rough sufficient lower bound, depending only on $m-s$;
using more information about $\pi_m \circ \dots \circ \pi_{s+1}$, we could provide some sharper lower bound.
 For simplicity, denote $r':=\max\{r'_{s-k}, \dots, r'_{s-1}\}$.
A  local calculation shows that, with respect to the blow-up $\pi_{t+1}:X_{t+1} \to X_t$, the multiplicity of any point in $\tilde{\mathscr{H}}_j \subset X_{t+1}$ can be at most the double of the maximal multiplicity of points in $\tilde{\mathscr{H}}_j \subset X_{t}$.
So we can choose for example
\[
N'_b > 2^{m-s}r'kn   \quad\text{ for all } b\in\{1,\dots,s-1\},
\]
meaning in particular that
\[
\ell_b > \frac{2^{m-s}r'kn}{r'_b}  \quad\text{ for all }  b\in\{s-k,\dots,s-1\}.
\]
One can easily verify that then  $N'_i >0$ as soon as $\pi_{s,i-1}(Z_i) \subset E_b$ for at least one $b<s$.

We show \ref{step5-2} by induction on $i$.
The case $i=s+1$ is clear using \eqref{formula N'_i}), since $\spl_{s+1} = \emptyset$ by Step \ref{imageofZ_i} and $\ct_{s+1}=\{s\}$.
Take now $i> s+1$.  We claim that
\[
\pi_{s,a-1}(Z_a) \not\subset \bigcup_{b<s}E_b\quad \text{for all } a \in\ct_i.
\]
Indeed, the inclusion $Z_i\subset E_a$ implies that $\pi_{s,i-1}(Z_i) \subset \pi_{s,a-1}(Z_a)$.

Then $N'_a=0$ for all $a \in\ct_i$ by induction.  And then $N'_i=0$, again using
\eqref{formula N'_i} and the fact that $\spl_{i} = \emptyset$ by Step \ref{imageofZ_i}.
\end{proof}

We may and will assume that  $N'_1, \dots, N'_{s-k-1}$ and $\ell_{s-k},\dots,\ell_{s-1}$ are chosen big enough as in Step \ref{general positivity}.
So, for $i>s$, if there exists $b<s$ such that $\pi_{s,i-1}(Z_i) \subset E_b$,  then $N_i \geq N'_i >0$.
 We still have to show that, if $N'_i=0$, then we can make appropriate choices such that $N_i>0$. From Step \ref{general positivity}, we know that in this case $\displaystyle\pi_{s,i-1}(Z_i) \not\subset \bigcup_{b<s}E_b$.
We start with the case $i=s+1$.

\begin{step}\label{casefirstinduction}
Assume that $N'_{s+1}=0$. 
Then there exists a linear polynomial $\poly_{s+1}\in\mathbb{Q}[k_j\mid j\in\zz_{s+1}]$,
with positive coefficients in all variables such that
\[
\alpha_{s+1}=\alpha_s+a_{ss}W_{s+1}
\]
and, if
\[
\frac{\alpha_s}{a_{ss}}<\grande<\frac{\alpha_s}{a_{ss}}+W_{s+1}(k_j\mid j\in\zz_{s+1}),
\]
then $N_{s+1}>0$.
\end{step}

\begin{proof}[Proof of Step{\rm~\ref{casefirstinduction}}]
Since by assumption $\alpha_{s+1}=\beta_{s+1}$, we have that
\begin{equation}\label{eq:minspecial}
N_{s+1}=\alpha_{s+1}-\min(\alpha_{s+1},\grande a_{s,s+1}).
\end{equation}
Note that $a_{s,s+1}=\nu_{s+1}(\mathcal{C}_s)=\nu_{s}(\mathcal{C}_s)=a_{ss}$ and (by Step \ref{casoalpha})
\[
\alpha_{s+1}=\alpha_s+\sum_{j\in\zz_{s+1}} k_j \mu_{s+1}(\mathcal{C}_j).
\]
Recall that $s+1\in\zz_{s+1}$, so the summation above is \lq not empty\rq.
Let us define
\begin{equation}\label{eq:zs1}
W_{s+1}(k_j\mid j\in\zz_{s+1})=\frac{1}{a_{ss}}\sum_{j\in\zz_{s+1}} k_j \mu_{s+1}(\mathcal{C}_j),
\end{equation}
which is clearly a linear function with positive rational coefficients
since these multiplicities are always positive. If
\[
\grande<\frac{\alpha_s}{a_{ss}}+W_{s+1}(k_j\mid j\in\zz_{s+1}),
\]
then $\grande a_{ss}<\alpha_{s+1}$ and $N_{s+1}>0$ using \eqref{eq:minspecial}.

The  inequality $\frac{\alpha_s}{a_{ss}}<\grande$ in the statement guaranteed that $N_s=0$ (see Case \ref{case2}). The point is that there are integer solutions for $\ell$ and the $k_j$ satisfying both inequalities. (Recall that $\alpha_s$ itself is also a polynomial in the $k_j$.)
\end{proof}

Let us define recursively
$\tilde{\zz}_{s+1}:=\zz_{s+1}$ and
\[
\tilde{\zz}_{i}:=\zz_i\cup\bigcup_{b\in\ct_i} \tilde{\zz}_{b} \qquad\text{ for } i>s+1.
\]

\begin{step}
\label{casezerogeneral}
Let $i>s$ such that $N'_i=0$. Then
there exists  $p_i \in \mathbb{Z}_{>0}$ and a linear polynomial $W_i\in\mathbb{Q}[k_j\mid j\in\tilde{\zz}_i]$, with positive coefficients
in all variables, such that
\begin{equation}\label{eq:alpha_i via alpha_s}
\alpha_i  = p_i (\alpha_s + a_{ss} W_i)
\end{equation}
and, if
\begin{equation}\label{eq:condition}
\frac{\alpha_s}{a_{ss}}<\grande<\frac{\alpha_s}{a_{ss}}+W_i(k_j\mid j\in\tilde{\zz}_i) ,  
\end{equation}
then $N_i>0$.
\end{step}

\begin{proof}[Proof of Step{\rm~\ref{casezerogeneral}}]
We proceed  by induction on $i$. The case $s+1$ corresponds to Step~\ref{casefirstinduction} where
$\tilde{\zz}_{s+1}=\zz_{s+1}$ and
$W_{s+1}$ is shown in \eqref{eq:zs1}.

Let now $i>s+1$.  Since $\spl_{i}=\emptyset$, we have by \eqref{formula N'_i} that $N'_a =0$  for all $a\in\ct_i$.
Then
\begin{equation*}
\alpha_i=\sum_{a\in\ct_i}\alpha_a + \sum_{j\in\zz_i} k_j \mu_i(C_j)
=\left(\sum_{a\in\ct_i}p_a\right)\alpha_s  + a_{ss}\sum_{a\in\ct_i} p_a  W_a+ \sum_{j\in\zz_i} k_j \mu_i(C_j),
\end{equation*}
where the first equality is Step \ref{casoalpha}, and the second one is by induction.
We define $p_i := \sum_{a\in\ct_i} p_a$ and
\[
W_i:= \frac{1}{p_i}\sum_{a\in\ct_i}p_a W_a + \frac{1}{p_ia_{ss}} \sum_{j\in\zz_i} k_j\mu_i(C_j) \in\mathbb{Q}[k_j\mid j\in\tilde{\zz}_i],
\]
which is a linear function with positive coefficients in all variables.
Moreover
\[
\alpha_i=p_i (\alpha_s +a_{ss} W_i).
\]
Finally, recall that $a_{si}=a_{ss}$ and that the condition $N'_i=0$ says that $\alpha_i=\beta_i$. Then we know from \eqref{eq:min} that $N_i>0$ if \eqref{eq:condition} holds.
\end{proof}
If we choose the coordinates of $\mathbf{k}$ big enough the intervals of solution
are of length $>1$, and then we can ensure the existence of suitable $\grande$.
As final conclusion we then indeed obtain a function $h$ satisfying the assertion of the theorem.


\begin{thebibliography}{10}

\bibitem{aba:13}
S.S. Abhyankar and E.~Artal, \emph{Algebraic theory of curvettes and
  dicriticals}, Proc. Amer. Math. Soc. \textbf{141} (2013), no.~12, 4087--4102.

\bibitem{aba:14}
\bysame, \emph{Analytic theory of curvettes and dicriticals}, Rev. Mat.
  Complut. \textbf{27} (2014), no.~2, 461--499.

\bibitem{AH:12}
S.S. Abhyankar and W.J. Heinzer, \emph{Existence of dicritical divisors}, Amer.
  J. Math. \textbf{134} (2012), no.~1, 171--192.

\bibitem{AL:11}
S.S. Abhyankar and I.~Luengo, \emph{Algebraic theory of dicritical divisors},
  Amer. J. Math. \textbf{133} (2011), no.~6, 1713--1732.

\bibitem{arn}
V.I. {Arnol\cprime d}, \emph{Singularities of fractions and the behavior of
  polynomials at infinity}, Tr. Mat. Inst. Steklova \textbf{221} (1998),
  48--68.

\bibitem{cae:10}
E.~{Casas-Alvero}, \emph{Approximate implicit equations of parameterized germs
  of plane curve}, JP J. Algebra Number Theory Appl. \textbf{19} (2010), no.~2,
  225--243.

\bibitem{clmn}
J.I. Cogolludo, T.~L\'aszl\'o, J.~{Mart\'in-Morales}, and A.~N\'emethi,
  \emph{Delta invariant of curves on rational surfaces {I}. {A}n analytic
  approach}, Commun. Contemp. Math. \textbf{24} (2022), no.~7, Paper No.
  2150052, 23.

\bibitem{glm}
S.M. {Guse\u{\i}n-Zade}, I.~Luengo, and A.~Melle, \emph{Zeta functions of germs
  of meromorphic functions, and the {N}ewton diagram}, Funktsional. Anal. i
  Prilozhen. \textbf{32} (1998), no.~2, 26--35, 95.

\bibitem{har}
R.~Hartshorne, \emph{Algebraic geometry}, Springer-Verlag, New York-Heidelberg,
  1977, Graduate Texts in Mathematics, No. 52.

\bibitem{lw}
{L\^e{} D.T.} and C.~Weber, \emph{A geometrical approach to the {J}acobian
  conjecture for {$n=2$}}, vol.~17, 1994, Workshop on Geometry and Topology
  (Hanoi, 1993), pp.~374--381.

\bibitem{vn:24}
A.~N\'emethi and W.~Veys, \emph{Filtrations associated with singularities},
  arXiv preprint math/2405.10898 (2024).

\bibitem{oku}
T.~Okuma, \emph{Another proof of the end curve theorem for normal surface
  singularities}, J. Math. Soc. Japan \textbf{62} (2010), no.~1, 1--11.

\bibitem{vy:07}
W.~Veys, \emph{Monodromy eigenvalues and zeta functions with differential
  forms}, Adv. Math. \textbf{213} (2007), no.~1, 341--357.

\end{thebibliography}

\def\cprime{$'$}
\providecommand{\bysame}{\leavevmode\hbox to3em{\hrulefill}\thinspace}
\providecommand{\MR}{\relax\ifhmode\unskip\space\fi MR }
\providecommand{\MRhref}[2]{%
  \href{http://www.ams.org/mathscinet-getitem?mr=#1}{#2}
}
\providecommand{\href}[2]{#2}

\end{document}